\documentclass[11pt,a4paper]{amsart}
\usepackage{a4wide}
\usepackage{amsfonts}
\usepackage{amscd}
\usepackage{amssymb}
\usepackage{amsthm}
\usepackage{amsmath}
\usepackage{latexsym}
\usepackage{dsfont}
\usepackage{graphicx}
\usepackage{subfigure}
\usepackage{placeins}
\usepackage{algorithm}
\usepackage{algpseudocode}

\usepackage{pgfplots}
\pgfplotsset{compat=1.13}
\usepackage{booktabs}
\usepackage{epstopdf}
\usepackage[labelfont={bf,sf},font={small},labelsep=space]{caption} 

\theoremstyle{plain}
\newtheorem{thm}{Theorem}[section]
\theoremstyle{plain}
\newtheorem{lem}[thm]{Lemma}

\newtheorem{cor}[thm]{Corollary}
\theoremstyle{definition}
\newtheorem{defi}[thm]{Definition}
\newtheorem{rem}[thm]{Remark}
\newtheorem{assumption}[thm]{Assumption}
\newtheorem{algo}[thm]{Algorithm}
\newtheorem{ex}[thm]{Example}
{%
\setcounter{enumi}{0}

\begin{enumerate}}%
{\end{enumerate} }
{%
\setcounter{enumi}{0}

\begin{enumerate}}%
{\end{enumerate} }

\newcommand{\D}{\ensuremath{\mathcal{D}}}

\newcommand{\eps}{\ensuremath{\varepsilon}}
\newcommand{\R}{\ensuremath{\mathbb{R}}}
\newcommand{\N}{\ensuremath{\mathbb{N}}}

\newcommand{\T}{\ensuremath{\mathbb{T}}}
\newcommand{\E}{\ensuremath{\mathbb{E}}}

\newcommand{\ga}{\alpha}
\newcommand{\gb}{\beta}
\newcommand{\gd}{\delta}

\renewcommand{\gg}{\gamma}
\newcommand{\gk}{\kappa}
\newcommand{\gl}{\lambda}
\newcommand{\go}{\omega}
\newcommand{\gs}{\sigma}
\newcommand{\gt}{\theta}

\newcommand{\gD}{\Delta}
\newcommand{\gG}{\Gamma}

\newcommand{\gO}{\Omega}

\newcommand{\cA}{\mathcal{A}}
\newcommand{\cB}{\mathcal{B}}

\newcommand{\cD}{\mathcal{D}}

\newcommand{\cL}{\mathcal{L}}
\newcommand{\cM}{\mathcal{M}}
\newcommand{\cN}{\mathcal{N}}

\newcommand{\cU}{\mathcal{U}}

\newcommand{\1}{\mathbf{1}}
\newcommand{\bP}{\mathbb{P}}

\makeatletter
\usepackage{color}

\makeatother

\newcommand{\be}{\begin{equation}}
\newcommand{\ee}{\end{equation}}
\numberwithin{equation}{section} \allowdisplaybreaks[1]

\newcommand{\KL}{{Karhunen-Lo\`{e}ve }}

\newcommand*{\invtr}{\ensuremath{^-~\!\!'}}

\definecolor{darkgreen}{rgb}{0,.6,0}

\title[Approximation and simulation of infinite-dimensional L\'evy processes]{
Approximation and simulation of Infinite-dimensional L\'evy processes
}
\thanks{%
The research leading to these results has received funding from the German Research Foundation (DFG) as part of the Cluster of Excellence in Simulation Technology (EXC 310/2) at the University of Stuttgart and by the Juniorprofessorship program of Baden--W\"urttemberg, and it is gratefully acknowledged. The authors would like to thank Denis Talay for the fruitful discussions and helpful comments, as well as for the remarks of the anonymous reviewers which led to a significant improvement of the manuscript. }

\author[Barth]{Andrea Barth}
\address[Andrea Barth]{\newline SimTech, University of Stuttgart\newline Allmandring 5b\newline 70569 Stuttgart, Germany}
\email[]{andrea.barth\@@mathematik.uni-stuttgart.de}

\author[Stein]{Andreas Stein}
\address[Andreas Stein]{\newline SimTech, University of Stuttgart\newline Allmandring 5b\newline 70569 Stuttgart, Germany}
\email[]{andreas.stein\@@mathematik.uni-stuttgart.de}

\begin{document}

\begin{abstract}
In this paper approximation methods for infinite-dimensional L\'evy processes, also called (time-dependent) L\'evy fields, are introduced. 
For square integrable fields beyond the Gaussian case, it is no longer given that the one-dimensional distributions in the spectral representation with respect to the covariance operator are independent. When simulated via a \KL expansion a set of dependent but uncorrelated one-dimensional L\'evy processes has to be generated. The dependence structure among the one-dimensional processes ensures that the resulting field exhibits the correct point-wise marginal distributions. To approximate the respective (one-dimensional) L\'evy-measures, a numerical method, called discrete Fourier inversion, is developed. For this method, $L^p$-convergence rates can be obtained and,
under certain regularity assumptions, mean square and $L^p$-convergence of the approximated field is proved. Further, a class of (time-dependent) L\'evy fields is introduced, where the point-wise marginal distributions are dependent but uncorrelated subordinated Wiener processes. For this specific class one may derive point-wise marginal distributions in closed form. Numerical examples, which include hyperbolic and normal-inverse 
Gaussian fields, demonstrate the efficiency of the approach.
\end{abstract}

\keywords{Infinite-dimensional L\'evy processes, \KL expansion, subordinated processes, stochastic partial differential equations}

\maketitle
\section{Introduction}\label{sec:Intro}

Uncertainty quantification plays an increasingly important role in a wide range of problems in the Engineering Sciences and Physics. Examples of sources of uncertainty are imprecise or insufficient measurements and noisy data. In the underlying dynamical system this is modeled via a stochastic operator, stochastic boundary conditions and/or stochastic data. As an example, to model subsurface flow more realistically the coefficients of an (essentially) elliptic equation are assumed to be stochastic. A common approach in the literature is to use (spatially) correlated random fields that are built from uniform distributions or colored log-normal fields. The resulting point-wise marginal distributions of the field are (shifted) normally, resp. log-normally distributed. Neither choice is universal enough to accommodate all possible types of porosity, especially not if fractures are incorporated (see~\cite{ZK04}). In some applications it might even be necessary that the point-wise marginal distribution of the (time-dependent) random field is a pure-jump process (see~\cite{BB14}). Here, we denominate by \emph{point-wise marginal distributions} the distributions resp. processes one obtains by evaluation of the random field at a fixed spatial point. On a note, these are in general the distributions that may be measured in applications.
 
In the case of a (time-dependent) Gaussian random field, the approximation and simulation via its \KL (KL) expansion is straightforward. Almost sure and $L^p$-convergence in terms of the decay of the eigenvalues has been shown for truncated KL-expansions in~\cite{BaLa12_2}.  
For infinite-dimensional L\'evy processes, also called \textit{L\'evy fields}, the approximation may still be attempted via the KL expansions: On a separable Hilbert space $(H,(\cdot,\cdot)_H)$ with orthonormal basis $(e_i,i\in\N)$, a square-integrable L\'evy field $L = (L(t)\in H,t\ge 0)$ admits the expansion 
\begin{equation*}
L(t)=\sum_{i\in\N} (L(t),e_i)_H e_i,
\end{equation*}
The sequence $((L(\cdot),e_i)_H, i\in N)$ consists of one-dimensional, real-valued L\'evy processes.
In contrast to the case of a Gaussian field, the one-dimensional processes $((L(t),e_i)_H,t\ge0)$ in the spectral representation are not independent but merely uncorrelated. 
If one were to use independent L\'evy processes, the resulting field would not have the desired point-wise marginal distributions and the KL expansion would, therefore, not converge to the desired L\'evy field.
To circumvent this issue, we approximate $L$ by truncating the series after finite number of terms and generate dependent but uncorrelated processes $((L(t),e_i)_H,t\ge0)$.

This entails, however, the simulation of one-dimensional L\'evy processes. A common way to do so, is to employ the so called \textit{compound Poisson approximation} (CPA) (see~\cite{AR01,F11,MS07,R97,S03} or the references therein).
Mean-square convergence results for the CPA are available in some cases, but require rather strong assumptions on the underlying process.
In addition, the obtained convergence rates are comparably low with respect to the employed time discretization, which implies that the CPA may not be suitable to sample processes involving computationally expensive components.
As one of the main contributions in this paper, we develop a novel approximation method for one-dimensional L\'evy processes.
This new approach, based on L\'evy bridge laws and Fourier inversion, addresses the abovementioned problems. 
We prove $L^p$- and almost surely convergence of the approximation under relatively weak assumptions and derive precise error bounds.
We show mean-square convergence of the approximation to a given infinite-dimensional L\'evy process by combining the Fourier inversion method with an appropriate truncation of the KL expansion. 

To obtain a set of dependent but uncorrelated one-dimensional processes, we utilize multi-dimensional time-changed Brownian motions.  
The underlying variance process is represented by a positive and increasing L\'evy process, a so-called \textit{subordinator}.
As a class of subordinated processes, we consider \textit{generalized hyperbolic (GH) L\'evy processes}, that are based on the generalized hyperbolic distribution and cover for example \textit{normal inverse Gaussian} (NIG) and \textit{hyperbolic} processes.
These processes are widely used in applications such as Mathematical Finance, Physics and Biology (see, for instance,~\cite{BN77,BB14,B81,E01,EK95}). With its fat-tailed distribution a GH-field may also be of value in the modeling of subsurface flows (see~\cite{ZK04}). 
For an overview on subordinated, Hilbert space-valued L\'evy processes we refer to \cite{BZ10,PR05}, where this topic is treated in a rather general setting. 
Among other subordinated Wiener processes, the construction of an infinite-dimensional NIG process can be found in \cite{BK15}. 
As a further contribution of this paper, we approximate the corresponding GH L\'evy fields via truncated KL expansions with dependent but uncorrelated GH-distributed one-dimensional processes and show that the approximation converges to an infinite-dimensional GH process.
From a simulatory point of view this entails the generation of a certain number of one-dimensional processes with a given set of parameters.
Conversely, we introduce a second approach, where we derive the dependence structure of the multi-dimensional GH process to obtain 
admissible sets of parameters such that the one-dimensional marginal GH processes are decorrelated and follow a desired distribution. 
Using the Fourier inversion method we are able to simulate GH fields efficiently, even if a large number of one-dimensional GH processes is necessary. 

This article is structured as follows: 
Section~\ref{sec:pre} contains preliminaries on L\'evy processes taking values in Hilbert spaces and the main convergence theorem for the approximation. 
In Section~\ref{sec:GIG_sim}, we present a new approach for the approximation of one-dimensional L\'evy processes by L\'evy bridge laws and prove $L^p$- and almost sure convergence. 
To be able to apply the algorithm in a very general setting, we introduce an extension by using Fourier inversion techniques and show how to control the $L^p$-error.
We proceed by investigating the class of GH L\'evy processes and state the necessary conditions for the approximated field to have point-wise GH distributed marginals.
In Section~\ref{sec:num}, we remark on some implementational details of the algorithm and conclude with NIG- and hyperbolic fields as numerical examples.

\section{Preliminaries}\label{sec:pre}
Throughout this paper, we consider a time interval $\T:=[0,T]$, with $T>0$, and a filtered probability space $(\gO,(\mathcal A_t,t\geq 0), \bP)$ satisfying the usual conditions. Let $(H,(\cdot,\cdot)_H)$ be a separable Hilbert space and $(H,\cB(H))$ a measurable space, where $\cB(H)$ denotes the Borel $\gs$-algebra on $H$.
A L\'evy process taking values in $(H,(\cdot,\cdot)_H)$ is defined as follows (see~\cite{PZ07}):
\begin{defi}
A $H$-valued stochastic process $L=(L(t),t\in\T)$ is called L\'evy process\footnote{In the case that $H$ is an infinite-dimensional Hilbert space, sometimes $L$ is also called \textit{L\'evy field} to have a clear distinction from finite-dimensional L\'evy processes.} if 
\begin{itemize}
 \item $L$ has stationary and independent increments,
 \item $L(0)=0$ $\bP$-almost surely and
 \item $L$ is stochastically continuous, i.e. for all $\eps>0$ and $t\in\T$ holds
 \begin{equation*}
  \lim\limits_{s\to t, s\in\T} \bP(||L(t)-L(s)||_H>\eps)=0.
 \end{equation*}
\end{itemize}                                    
\end{defi}

The characteristic function of a L\'evy process is then given by the \textit{L\'evy-Khintchine formula}:
\begin{equation*}
 \E[\exp(i(h,L(t))_H)]=\exp(t\Psi_L(h)),\quad \text{for }h\in H,
\end{equation*}
where the exponent is of the form
\begin{equation}\label{eq:LKF}
 \Psi_L(h)=i(\iota_H,h)_H-\frac{1}{2}(\Sigma_H h,h)_H+\int_H \exp(i(h,y)_H)-1-i(h,y)_H\1_{||y||_H<1}\nu_H(dy)
\end{equation}
(see ~\cite[Thm. 4.27]{PZ07}).
In Eq.~\eqref{eq:LKF}, $\iota_H\in H$, $\Sigma_H$ is a symmetric, non-negative and nuclear operator on $H$ and $\nu_H:\cB(H)\to [0,\infty)$ is a non-negative, $\gs$-finite measure on $\cB(H)$ satisfying
\begin{equation*}
\nu_H(\{0\})=0\quad\text{and}\quad\int_H\min(1,||y||_H^2)\,\nu(dy)<\infty.
\end{equation*}
The triplet $(\iota_H,\Sigma_H,\nu_H)$ is unique for every L\'evy process $L$ and called the \textit{characteristic triplet}.
For the special case of a one-dimensional L\'evy process $\ell=(\ell(t),t\in\T)$, the L\'evy-Khintchine formula simplifies to
\begin{equation}\label{eq:LKD}
\E[\exp(iu\ell(t))]=\exp\left(t(\iota ui -\frac{\gs^2}{2}u^2+\int_\R\exp(iuy)-1-iuy\1_{|y|<1}d\nu(y)) \right),\quad u\in\R ,
\end{equation}
where $\iota\in\R$, $\gs^2>0$ and $\nu$ is a ($\gs$-finite) measure on $\cB(\R)$ satisfying
\begin{equation*}
\nu(\{0\})=0\quad\text{and}\quad\int_\R\min(1,y^2)\nu(dy)<\infty,
\end{equation*}
see for instance~\cite{BMR12} or~\cite{S99}.
The notation $\ell$ is introduced for finite-dimensional L\'evy processes to have a clear distinction from the (possibly infinite-dimensional) L\'evy process $L$.

If $W$ is a $H$-valued L\'evy field with characteristic triplet $(0,\Sigma_H,0)$, then $W$ is called \textit{$\Sigma_H$-Wiener process}. If, further, $\Sigma_H$ is symmetric, non-negative and nuclear (see Eq.~\eqref{eq:LKF}) it admits, by the Hilbert-Schmidt theorem, the spectral decomposition
\begin{equation*}
 \Sigma_H \hat e_i = \hat\rho_i \hat e_i.
\end{equation*}
Here, $((\hat\rho_i,\hat e_i),i\in\N)$ is the sequence of eigenpairs of $\Sigma_H$, where the eigenvalues $\hat\rho_i$ are positive with zero as their only accumulation point and the sequence $(\hat e_i,i\in\N)$ forms an orthonormal basis of $H$.
For convenience, we assume that the sequence of eigenvalues $(\hat\rho_i,i\in\N)$ is given in decaying order.
The $\Sigma_H$-Wiener process $W$ admits then a unique expansion (also called Karhunen--Lo\`eve expansion)
\begin{equation*}
 W(t) = \sum_{i\in\N} \sqrt{\hat\rho_i} \hat e_i w_i(t),
\end{equation*}
where $(w_i,i\in\N)$ is a sequence of independent, one-dimensional, real-valued Brownian motions. 
An obvious way to approximate $W$ is, therefore, given by the truncated series 
\begin{equation*}
 W_N(t) := \sum_{i=1}^N \sqrt{\hat\rho_i} \hat e_i w_i(t).
\end{equation*}
It can be shown that the approximations $(W_N,N\in\N)$ converge in $L^2(\gO;H)$ and almost surely to the $\Sigma_H$-Wiener process $W$ (see for instance~\cite{BL12}).
For the approximation of general (non-continuous) processes $L$, we aim to apply a similar approach. We assume $L$ is square-integrable, as otherwise $L$ does not admit a KL expansion. For series representations of cylindrical L\'evy processes we refer to~\cite{AR10}, KL expansions for white noise L\'evy fields may be found in~\cite{DHP12}.
A $H$-valued stochastic process $(L(t),t\in\T)$ is said to be square-integrable if $||L(t)||_{L^2(\gO;H)}:=\E(||L(t)||_H^2)<+\infty$ for all $t\in\T$. 
Obviously, mean-square convergence can only be well-defined for processes with his property. 
\begin{thm}(\cite[Theorem 4.44]{PZ07})\label{thm:covQ}
 Let $L$ be a square-integrable L\'evy process on $H$. Then there exists a $m\in H$ and a non-negative, symmetric trace class operator $Q$ on $H$ such that for all $h_1,h_2\in H$ and $s,t\in(0,T]$
 \begin{itemize}
  \item $\E((L(t),h_1)_H)=t(m,h_1)_H$,
  \item $\E((L(t)-mt,h_1)_H(L(s)-ms,h_2)_H)=\min(t,s)(Qh_1,h_2)_H$
  \item $\E(||L(t)-mt||_H^2)=t\: tr(Q)$,
 \end{itemize}
where $tr(Q)$ denotes the trace of $Q$. The operator $Q$ is also called covariance operator of $L$ and $m$ is called mean.
\end{thm}
Note that $Q$ in Theorem~\ref{thm:covQ} is not necessarily equal to the operator $\Sigma_H$ from the L\'evy-Khintchine formula~\eqref{eq:LKF}.
They are only equal if the measure $\nu_H$ is zero, meaning the process $L$ has no ``jump component''\footnote{This results in $L$ being a drifted $H$-valued Gaussian process.}.
The operator $Q$ admits a spectral decomposition with a sequence $((\rho_i,e_i),i\in\N)$ of orthonormal eigenpairs with non-negative eigenvalues.
Thus, $L$ has the spectral expansion
\begin{equation*}
L(t)=\sum_{i\in\N}(L(t),e_i)_He_i, 
\end{equation*}
where the one-dimensional L\'evy processes $((L(t),e_i)_H,t\in\T)$ are not independent, but merely uncorrelated (see~\cite[Section 4.8.2]{PZ07}).
For the approximation of $L$ we employ one-dimensional L\'evy processes $(\sqrt{\rho_i}\ell_i, i\in\N)$, so that $\sqrt{\rho_i}\ell_i(t)$ is equal to $(L(t),e_i)_H$ in distribution for all $t\in\T$ and all $i\in\N$, and define, for $N\in\N$, the truncated sum
\begin{equation*}
L_N(t):=\sum_{i=1}^N\sqrt{\rho_i} e_i\ell_i(t).
\end{equation*}
If the spectral basis $((\sqrt{\rho_i}e_i),i\in\N)$ of $H$ is given, the approximation of $L$ by $L_N$ reduces to the simulation of dependent but uncorrelated one-dimensional processes $\ell_i$.
In general, the processes $\ell_i$ have \textit{infinite activity}, i.e. $\bP$-almost all paths of the process $(\ell_i(t),t\in\T)$ have an infinite number of jumps in every compact time interval.
Popular examples of L\'evy processes with infinite activity are normal inverse Gaussian processes or hyperbolic processes, see~\cite{E01}.  
As it is not possible to simulate infinitely many jumps, we need to find a suitable approximation $\widetilde\ell_i$ of $\ell_i$ and define
\begin{equation*}
\widetilde L_N(t) := \sum_{i=1}^N\sqrt{\rho_i} e_i \widetilde\ell_i(t).
\end{equation*}
In the following, we derive a condition on the approximations $\widetilde\ell_i$ that ensures convergence of $\widetilde L_N$ to $L$ in $L^2(\gO;H)$ uniformly on $\T$.
Throughout this paper, we construct approximations  $\widetilde\ell_i$ from a skeleton of discrete realizations at fixed and equidistant points in $\T$.
To this end, we introduce, for given $n\in\N$, a time increment $\gD_n:=T/2^n$ and the set $\Xi_n := \{t_j:=j\Delta_n, \; j=0,\ldots,2^n\}$. By $\widetilde\ell_i^{(n)}$ we denote some piecewise-constant c\`adl\`ag approximation of the process $\widetilde\ell_i$ (for a construction see Section~\ref{sec:GIG_sim}).

\begin{thm}\label{thm:H_error}
Let $L=(L(t),t\in\T)$ be a square-integrable, $H$-valued L\'evy process. 
The covariance operator $Q$ of $L$ admits a spectral decomposition by a sequence of (orthonormal) eigenpairs $((\rho_i,e_i),i\in\N)$.
Assume that, for $n\in\N$, there exists a sequence of approximations $(\widetilde\ell_i^{(n)},i\in\N)$ of the one-dimensional processes $(\ell_i,i\in\N)$ on the interval $\T$, such that the $L^2(\gO;\R)$-approximation error can be bounded by 
\begin{equation}\label{eq:C_ell}
\sup_{t\in\T}\E(|\ell_i(t)-\widetilde\ell_i^{(n)}(t)|^2)\le C_\ell\gD_n,
\end{equation}
where the constant $C_\ell>0$ is independent of $i$.
If, for all $i\in\N$, the processes $\sqrt\rho_i\ell_i$ are in distribution equal to $(L(\cdot),e_i)_H$ then the sequence of approximations $(\widetilde L_N(t), N\in\N)$ converges in mean-square-sense to $L(t)$, for each $t\in\T$, and the error is bounded by
\begin{equation*} 
\sup_{t\in\T}\E(||L(t)-\widetilde L_N(t)||^2_H)^{1/2}\le \big(T\sum_{i=N+1}^\infty\rho_i\big)^{1/2}+\big(C_\ell\gD_n\sum_{i=1}^N\rho_i\big)^{1/2}.
\end{equation*}
\end{thm}
\begin{proof}
We may assume without loss of generality that the process $L$ has zero mean. 
Using the triangle inequality, the error term $\E(||L(t)-\widetilde L_N(t)||^2_H)$ can be split into 
\begin{equation*} 
\E(||L(t)-\widetilde L_N(t)||^2_H)^{1/2}\le\E(||L(t)-L_N(t)||^2_H)^{1/2}+\E(||L_N(t)-\widetilde L_N(t)||^2_H)^{1/2}.
\end{equation*}
The square-integrability of $L$ guarantees that $Q$ is trace class and has positive eigenvalues, i.e. $tr(Q)=\sum_{i\in\N}\rho_i<+\infty$. 
$L(t)$ has covariance $tQ$, which yields for the first error term 
\begin{align*}
\E(||L(t)-L_N(t)||^2_H)&=\E(||L(t)||_H^2)+\E(||L_N(t)||_H^2)-2\E((L(t),L_N(t))_H)\\
&= t \;tr(Q)+\E\Big(\sum_{i,j=1}^N\big( (L(t),e_i)_He_i,(L(t),e_j)_He_j \big)_H\Big)\\
&\quad-2\E\big(\sum_{i=1}^N(L(t),(L(t),e_i)_He_i)_H\big)\\
&=t\sum_{i=1}^\infty\rho_i+\sum_{i=1}^N\E\big((L(t),e_i)_H^2\big)-2\sum_{i=1}^N\E\big((L(t),e_i)_H^2\big)\\
&=t\sum_{i=1}^\infty\rho_i-\sum_{i=1}^N\E\big((L(t),e_i)_H^2\big).
\end{align*}
With Theorem~\ref{thm:covQ} we obtain
\begin{equation*}
\E((L(t),e_i)_H^2)=t(Qe_i,e_i)_H=t\rho_i,
\end{equation*}
and hence
\begin{equation*}
\sup_{t\in\T}\E(||L(t)-L_N(t)||^2_H)= \sup_{t\in\T} t\sum_{i=N+1}^\infty\rho_i = T\sum_{i=N+1}^\infty\rho_i.
\end{equation*}
As $Q$ is a trace class operator, the sum on the right hand side becomes arbitrary small as $N\to\infty$.
This implies that $L_N$ converges in $L^2(\gO;H)$ uniformly on $\T$ to $L$. 

For the second error term, we derive with the assumption that $\sqrt\rho_i\ell_i \stackrel{\cL}{=} (L(\cdot),e_i)_H$ for all $i\in\N$ and Ineq.~\eqref{eq:C_ell}
\begin{align*}
\sup_{t\in\T}\E(||L_N(t)-\widetilde L_N(t)||^2_H)
&=\sup_{t\in\T}\sum_{i,j=1}^N\E\big(\sqrt{\rho_i\rho_j}(\ell_i(t)-\widetilde\ell_i^{(n)}(t))(\ell_j(t)-\widetilde\ell_j^{(n)}(t))(e_i,e_j)_H\big)\\
&=\sum_{i=1}^N\rho_i||e_i||_H^2\sup_{t\in\T}\E(|\ell_i(t)-\widetilde\ell_i^{(n)}(t)|^2)\le C_\ell\gD_n\sum_{i=1}^N\rho_i,
\end{align*}
which proves the claim. Above and for the remainder of the paper we express equality in distribution by the relation $\stackrel{\cL}{=}$. 
\end{proof}
\begin{rem}\label{rem:trunc}
Theorem~\ref{thm:H_error} states that the approximation $\widetilde L_N$ converges in $L^2(\gO;H)$ to $L$ uniformly on $\T$, 
for $N\to\infty$ and in the case that Ineq.~\eqref{eq:C_ell} holds with a constant $C_\ell$ in the limit $\gD_n\to 0$. 
We may equilibrate both error contributions by choosing $N\in\N$ such that
\begin{equation}\label{eq:trunc} 
T\sum_{i=N+1}^\infty\rho_i \approx C_\ell\gD_n\sum_{i=1}^N\rho_i.
\end{equation}
The sum of the eigenvalues, $tr(Q)$, is often known a priori (for example if $Q$ is a covariance operator of the Mat\'ern class, see Section~\ref{sec:num}). Then, only the first $N$ eigenvalues have to be determined until Eq.~\eqref{eq:trunc} is fulfilled. Further, optimal values for $\gD_n$ and $N$ may be chosen for given $C_\ell$ and $(\rho_i,i\in\N)$.
\end{rem}
Theorem~\ref{thm:H_error} may be generalized in an $L^p$-sense (the supremum is omitted for simplicity).
\begin{cor}
Let the assumptions of Theorem~\ref{thm:H_error} be fulfilled and, for $p\geq2$, $\E(||L(t)||^p)<+\infty$ for each $t\in\T$,
$\sum_{i\in\N}\rho_i^{p/2}<+\infty$ and 
\begin{equation*}
\E(|\ell_i(t)-\widetilde\ell_i^{(n)}(t)|^p)\le C_{p,\ell}\gD_n,
\end{equation*}
for some $C_{p,\ell}>0$ independent of $i$. Then, the $L^p(\Omega;H)$-error is bounded by
\begin{align*} 
\E(||L(t)-\widetilde L_N(t)||^p_H)^{1/p}&\le \big(\sum_{i>N} \rho_i\big)^{1/2-1/p}\Big(\sum_{i>N} \rho_i^{p/2} \E(|\ell_i(t)|^p)\Big)^{1/p}\\
&\quad+ (\sum_{i=1}^N \rho_i\big)^{1/2-1/p}\big(C_{\ell,p}\gD_n\sum_{i=1}^N \rho_i^{p/2}\Big)^{1/p}.
\end{align*}
\end{cor}

\begin{proof}
The proof follows closely the one of Theorem~\ref{thm:H_error}. We split the error into
\begin{equation*}
 \E(||L(t)-\widetilde L_N(t)||^p_H)^{1/p}\le \E(||L(t)-L_N(t)||^p_H)^{1/p}+\E(||L_N(t)-\widetilde L_N(t)||^p_H)^{1/p}.
\end{equation*}
For the first term follows 
\begin{align*}
 \E(||L(t)-L_N(t)||_H^p)=\E((||L(t)-L_N(t)||_H^2)^{p/2})=\E\Big( \big(\sum_{i>N} (L(t),e_i)_H^2\big)^{p/2}\Big).
\end{align*}
Since $\E(||L(t)||^p)<\infty$, $L_N$ converges to $L(t)$ in $L^p(\gO;H)$ by the Monotone Convergence Theorem. 
Moreover, using $(L(t),e_i)\stackrel{\cL}{=}\sqrt{\rho_i}\ell_i(t)$ and Jensen's inequality we may bound the above error via
\begin{align*}
 \E(||L(t)-L_N(t)||_H^p)&=\E\Big( \big(\sum_{i>N} \rho_i \ell_i(t)^2 \big)^{p/2}\Big)\\
 &\leq\big(\sum_{i>N} \rho_i\big)^{p/2-1}\sum_{i>N} \rho_i^{p/2} \E(|\ell_i(t)|^p),
\end{align*}
where we have used that $\E(|(L(t),e_i)_H|^p)=\rho_i^{p/2} \E(|\ell_i(t)|^p)$ and $\E(||L(t)||^p)<+\infty$.
Compared to the case $p=2$ with $\E(\rho_i^{p/2} |\ell_i(t)|^p)=\rho_i$, one needs additional assumptions on the $p$-th moment of $\ell_i$ to obtain an explicit bound.
In a similar fashion, the second error contribution is then bounded by  
\begin{align*}
 \E(||L_N(t)-\widetilde L_N(t)||_H^p)&=\E\Big(\big(\sum_{i=1}^N \rho_i|\ell_i(t)-\widetilde\ell_i^{(n)}(t)|^2\big)^{p/2}\Big)\\
 &\le \big(\sum_{i=1}^N \rho_i\big)^{p/2-1}\sum_{i=1}^N \rho_i^{p/2} \E(|\ell_i(t)-\widetilde\ell_i^{(n)}(t)|^p)\\
 &\le C_{\ell,p}\gD_n\big(\sum_{i=1}^N \rho_i\big)^{p/2-1}\sum_{i=1}^N \rho_i^{p/2}.
\end{align*}
\end{proof}

By a Borel--Cantelli-type argument almost sure convergence follows from Theorem~\ref{thm:H_error}.
\begin{cor}
	Let the assumptions of Theorem~\ref{thm:H_error} hold and the eigenvalues of $Q$ fulfill
	$\sum_{i\in\N} \rho_i (i-1)<+\infty$. 
	If for each $N\in\N$, $n(N)\in\N$ is chosen such that 
	\begin{equation*}
	\gD_{n(N)} \le \frac{T\sum_{i>N}\rho_i}{C_\ell \sum_{i=1}^N\rho_i}, \quad N\in\N,
	\end{equation*}
	(see Remark~\ref{rem:trunc}) the approximated L\'evy process $\widetilde L_N$ converges almost surely to $L$ in $H$ as $N\to\infty$, where the convergence is uniform in $\T$.
\end{cor}
\begin{proof}
	By Markov's inequality and Theorem~\ref{thm:H_error}, we obtain for any $\eps>0$ and $t\in\T$
	\begin{equation*}
	\bP(||L(t)-\widetilde L_N(t)||_H>\eps)\le\frac{\E(||L(t)-\widetilde L_N(t)||_H^2)}{\eps^2}
	\le \frac{1}{\eps^2} \Big(\big(T\sum_{i>N}\rho_i\big)^{1/2}+\big(C_\ell\gD_{n(N)}\sum_{i=1}^N\rho_i\big)^{1/2}\Big)^2.
	\end{equation*}	
	With $\gD_{n(N)}$ as above this yields
	\begin{equation*}
	\sum_{N\in\N}\bP(||L(t)-\widetilde L_N(t)||_H>\eps)\le\frac{4T}{\eps^2}\sum_{N\in\N}\sum_{i>N}\rho_i
	=\frac{4T}{\eps^2}\sum_{i>N}\rho_i(i-1)<\infty.
	\end{equation*}
	The claim follows by the Borel-Cantelli Lemma and by the fact that the sum on the right hand side in the inequality is independent of $t$. 
\end{proof}

For the approximation
\begin{equation*}
\widetilde L_N(t)=\sum_{i=1}^N\sqrt{\rho_i} e_i \widetilde\ell_i(t)
\end{equation*}
of $L$, we required that the one-dimensional L\'evy processes $(\widetilde\ell_i,i=1,\dots,N)$ are uncorrelated but not independent.
Several questions may arise regarding this truncated sum: 
\begin{enumerate}
 \item[1.] How can we efficiently simulate suitable one-dimensional approximations $\widetilde\ell_i$ of $\ell_i$ and determine the constant $C_\ell$ to apply Theorem~\ref{thm:H_error}?
 \item[2.] Is $L_N$ again a L\'evy field for arbitrary one-dimensional processes $(\ell_i,i\in\N)$ and can the point-wise marginal distribution of $L_N(t)$ for a given spectral basis $((\sqrt\rho_ie_i),i\in\N)$ and fixed $N\in\N$ be determined?  
 \item[3.] Is it possible to construct $L_N$ in a way such that its point-wise marginal processes follow a desired distribution?
\end{enumerate}
In the next chapter we address the first question and present a novel approach for the approximation of arbitrary one-dimensional L\'evy processes $\ell_i$. 
We derive explicit error bounds and convergence results in $L^p(\gO;\R)$, hence we are able to determine $C_\ell$ or at least bound this constant from above. 
The last two questions on the distribution properties of $L_N$ are then investigated in Section~\ref{sec:GH} for an important subclass of L\'evy fields.
We discuss distributional features of $L_N$ so as to use the approximation methodology developed in Section~\ref{sec:GIG_sim} efficiently to draw samples of the field $\widetilde L_N$.

\section{Simulation of L\'evy processes by Fourier inversion}\label{sec:GIG_sim}
The simulation of an arbitrary one-dimensional L\'evy process $\ell=(\ell(t),t\in\T)$ is not straightforward, as sufficiently many discrete realizations of $\ell$ in $\T$ are needed and the distribution of the increment $\ell(t+\gD_n)-\ell(t)$ for some small time step $\gD_n>0$ is not explicitly known in general.
A well-known and common way to simulate a L\'evy process with characteristic triplet $(\iota,\sigma^2,\nu)$ (see Equation \eqref{eq:LKD})
is the \textit{compound Poisson approximation} (CPA) suggested in~\cite{R97} and~\cite{S03}. 
All jumps of the process larger than some $\eps>0$ are approximated by a sum of independent compound Poisson processes and the small jumps by their expected values resp. by a Brownian motion. For details and convergence theorems of this method we refer to \cite{AR01,R97,S03}.
Although the CPA is applicable in a very general setting, in the sense that only the triplet $(\iota,\sigma^2,\nu)$ has to be known for simulation, it has several drawbacks. 
It is possible to show that the CPA converges under certain assumptions in distribution to a L\'evy process with characteristic triplet $(\iota,\sigma^2,\nu)$, 
and even strong error rates for CPA-type approximation schemes are given, for instance in \cite{DHP12, F11, MS07}. 
The derived $L^p$-error rates are, however, rather low with respect to the time discretization, only available for $p\le2$ and/or require strong assumptions on the moments of the L\'evy measure $\nu$.
Furthermore, if the cumulated density function (CDF) of $\nu$ is unknown, numerical integration with respect to $\nu$ is necessary. 
Evaluating the density of $\nu$ at sufficiently many points to obtain a good approximation might be time consuming, especially if this involves computationally expensive components (e.g. Bessel functions).
It is, further, a-priori not clear how to discretize the measure $\nu$ (we refer to a discussion on this matter in~\cite[Chapter 8]{S03}).
One could choose for example equidistant or equally weighted points, but this choice might have a significant impact on the precision and the speed of the simulation, and is impossible to be assessed beforehand. 
The disadvantages of the CPA method motivate the development of an alternative methodology.\\ 
In the following, we introduce a new sampling approach which approximates the process $\ell$ by a refining sequence of piecewise constant c\`adl\`ag processes $(\overline{\ell}^{(n)},n\in\N)$. 
We show its asymptotic convergence in $L^p(\gO;\R)$-sense and almost surely. This approximation suffers from the fact that the necessary conditional densities from which we have to sample are not available for many L\'evy processes.
For a given refinement parameter $n$, we develop, therefore, an algorithm to sample an approximation $\widetilde{\ell}^{(n)}$ of $\overline{\ell}^{(n)}$ for which the resulting error may be bounded again in $L^p(\gO;\R)$-sense.
This technique is based on the assumption that the characteristic function of $\ell$ is available in closed form, which is true for a broad class of L\'evy processes.
We exploit this knowledge by so-called \emph{Fourier inversion} to draw samples of the process' increments over an arbitrary large time step $\gD_n>0$. 
In Section~\ref{sec:num}, we then apply the described method to simulate GH L\'evy fields.

\subsection{A piecewise constant approximation of $\ell$}
Throughout this chapter, we consider a one-dimensional L\'evy process $\ell=(\ell(t),t\in\T)$ with characteristic function $\phi_\ell:\R\to\mathbb C$.
For any $t\in\T$, we denote by $F_t$ the CDF of $\ell(t)$ and by $f_t$ the corresponding density function, provided that $f_t$ exists. 
Note that in this case $F_t$ and $f_t$ belong to the probability distribution with characteristic function $(\phi_\ell)^t$.
To obtain a refining scheme of approximations of $\ell$, we introduce a sampling algorithm for $\ell$ based on the construction of \textit{L\'evy bridges}.
In our context, a L\'evy bridge is the stochastic process $(\ell(t)|t\in(t_1,t_2))$ pinned to given realizations of the boundary values $\ell(t_1)$ and $\ell(t_2)$ for $0\le t_1<t_2\le T$.  
It has been shown, for instance in \cite[Proposition 2.3]{HHM11}, that these bridges are Markov processes.
Assuming that the density $f_t$ exists for every $t\in\T$ (see also Remark~\ref{rem:2pi}), the distribution of the increment $\ell(t)-\ell(t_1)$ conditional on $\ell(t_2)$ is well-defined whenever $f_{t_2-t_1}(\ell(t_2))\in(0,+\infty)$.
Its density function is then given as  
\begin{equation}\label{eq:cond_dens}
f^{t_1,t_2}_{t}(x):=\frac{f_{t-t_1}(x)f_{t_2-t}(\ell(t_2)-x)}{f_{t_2-t_1}(\ell(t_2))},
\end{equation}
with conditional expectation $\E(\ell(t)|\ell(t_1),\ell(t_2))=\frac{\ell(t_2)-\ell(t_1)}{t_2-t_1}(t-t_1)$ (see \cite{HHM11},\cite{MY05}).
This motivates the following sampling algorithm for a piecewise constant approximation of $\ell$:
\begin{algo}\label{algo:bridge}
 Let $n\in\N$ and generate a sample of the random variable $\mathcal X_{0,1}$ with density $f_T$. Set $\mathcal X_{0,0}:=0$, $i:=1$ and $\gD_0:=T$.
  \begin{algorithmic}[1]
 \While{$i\le n$}
 \State Define $\gD_i=\frac{T}{2^i}$.
 \For{$j=0,2,\dots,2^i$}
 \State Set $\mathcal X_{i,j}=\mathcal X_{i-1,j/2}$.
 \EndFor
 \For{$j=1,3,\dots,2^i-1$}
 \State Generate the (conditional) increment $\mathcal X_{i,j}-\mathcal X_{i,j-1}$ within $[\frac{(j-1)T}{2^{(i-1)}},\frac{jT}{2^{(i-1)}}]$.
 \State That is, sample the random variable $X:\gO\to\R$ with density 
\State \begin{equation*}
 x\mapsto\frac{f_{\gD_i}(x)f_{\gD_i}(\mathcal X_{i,j+1}-x)}{f_{\gD_{i-1}}(\mathcal X_{i,j+1})}
\end{equation*}
\State and set $\mathcal X_{i,j}:=X+\mathcal X_{i,j-1}$
 \EndFor
 \State $i=i+1$ 
 \EndWhile
\end{algorithmic}
Define the piecewise constant process $\overline\ell^{(n)}(t):=\mathcal X_{n,2^n}\1_{\{T\}}(t)+\sum\limits_{j=1}^{2^n}\mathcal X_{n,j-1}\1_{\{[(j-1)T/2^n,jT/2^n)\}}(t)$.
\end{algo}
Eventually, the sequence $(\overline\ell^{(n)},n\in\N)$ of c\`adl\`ag processes admits a pointwise limit in $L^p(\gO;\R)$ which corresponds to the process $\ell$:
\begin{thm}\label{thm:bridge}
 Let $\phi_\ell$ be a characteristic function of an infinitely divisible probability distribution.
 For any $t\in\T$, assume the probability density $f_t$ corresponding to $(\phi_\ell)^t$ exists.
 Further, for $n\in\N$, let $\overline\ell^{(n)}$ be the process generated by Algorithm~\ref{algo:bridge} and $f_t$ on $(\gO,(\cA_t,t\geq0),\bP)$.
 If $\int_\R|x|^pf_1(x)dx<\infty$ for some $p\in[1,\infty)$, then 
 \begin{equation*}
  \lim_{n\to\infty}\E(|\overline\ell^{(n)}(t)-\ell(t)|^p)=0,
 \end{equation*}
  where $\ell$ is a L\'evy process with characteristic function $\phi_\ell$ on $(\gO,(\cA_t,t\geq0),\bP)$.
\end{thm}
\begin{proof}
For any $n\in\N$ and $t\in\T$ we have that
\begin{align*}
 &\E[|\overline\ell^{(n+1)}(t)-\overline\ell^{(n)}(t)|^p]\\
 =&\E\Big(\big|\sum\limits_{j=1}^{2^{n+1}}\mathcal X_{n+1,j-1}\1_{\{[(j-1)T/2^{n+1},jT/2^{n+1})\}}(t)-\sum\limits_{j=1}^{2^n}\mathcal X_{n,j-1}\1_{\{[(j-1)T/2^n,jT/2^n)\}}(t)\big|^p\Big)\\
 =&\E\Big(\big|\sum\limits_{j=1}^{2^n}(\mathcal X_{n+1,2j-1}-\mathcal X_{n+1,2j-2})\1_{\{[(2j-1)T/2^{n+1},2jT/2^{n+1})\}}(t)\big|^p\Big)\\
\end{align*}
Since the increments $\mathcal X_{n+1,j+1}-\mathcal X_{n+1,j}$ are i.i.d. with characteristic function $(\phi_{\ell})^{T/2^{n+1}}$ by construction, this yields
$$\E[|\overline\ell^{(n+1)}(t)-\overline\ell^{(n)}(t)|^p]\le C_{\ell,T}2^{-n-1}\int_\R|x|^pf_1(x)dx=C_{\ell,T,p}2^{-n-1},$$
where $C_{\ell,T}$ resp. $C_{\ell,T,p}$ are positive constants that are independent of $n$.
Hence, for any $m,n\in\N$ with $m>n$ it follows
$$\E[|\overline\ell^{(m)}(t)-\overline\ell^{(n)}(t)|^p]^{1/p}\le C_{\ell,T,p}^{1/p}\sum_{i=n+1}^m2^{-i/p}=C_{\ell,T,p}^{1/p}\frac{2^{-n/p}-2^{-m/p}}{2^{1/p}-1},$$
meaning that $(\overline\ell^{(n)}(t),n\in\N)$ is a $L^p(\gO;\R)$-Cauchy sequence and, therefore, admits a limit.
The characteristic function of $\overline\ell^{(n)}(t)$ is given by $(\phi_\ell)^{\lfloor t2^n/T\rfloor T/2^n}\stackrel{n\to\infty}{\to}(\phi_\ell)^t$.
The claim follows since the distribution with characteristic function $\phi_\ell$ is infinitely divisible, hence the limit process $\ell=(\ell(t),t\in\T)$ is in fact a L\'evy process.
\end{proof} 
\begin{cor}
 Under the assumptions of Theorem~\ref{thm:bridge} with $p=1$, $\overline\ell^{(n)}$ converges to $\ell$ $\bP$-almost surely as $n\to+\infty$, uniformly in $\T$.
\end{cor}
\begin{proof}
 For any $t\in\T$ and $\eps>0$, we get by Markov's inequality
 \begin{equation*}
  \bP(|\overline\ell^{(n)}(t)-\ell(t)|)\le\frac{\E(|\overline\ell^{(n)}(t)-\ell(t)|)}{\eps}\le\frac{C_{\ell,T,p}}{\eps}\sum_{i=n}^\infty2^{-i}=\frac{C_{\ell,T,p}2^{-n+1}}{\eps}.
 \end{equation*}
The claim then follows by the Borel-Cantelli Lemma since 
 \begin{equation*}
  \sum_{n=1}^\infty\bP(|\overline\ell^{(n)}(t)-\ell(t)|)\le\frac{2C_{\ell,T,p}}{\eps}\sum_{n=1}^\infty 2^{-n}<+\infty.
 \end{equation*}
\end{proof}

Although Algorithm~\ref{algo:bridge} has convenient properties in terms of convergence, it may only be applied for a small class of L\'evy processes.
For a general L\'evy process $\ell$, the conditional densities in Eq.~\eqref{eq:cond_dens} will be unknown and thus simulating from this distributions is impossible. 
A few exceptions where ``bridge sampling'' of L\'evy processes is feasible include the inverse Gaussian (\cite{Rw03}) and the tempered stable process (\cite{KK16}).
However, if we consider a fixed parameter $n$, sampling from the bridge distributions is equivalent to the following algorithm:

\begin{algo}\label{alog:mod}
	~
  \begin{algorithmic}[1] 
\State For $n\in\N$, fix $\gD_n,\Xi_n$ as in Section~\ref{sec:pre} and generate $2^n$ i.i.d random variables $ X_1,\dots,X_{2^n}$ with density $f_{\gD_n}$.
\State Set $\ell^{(n)}(t)=0$ if $t\in[0,t_1)$, $\ell^{(n)}(t)=\sum_{k=1}^jX_k$ if $t\in[t_j,t_{j+1})$ for $j=1,\dots,2^n-1$ and $\ell^{(n)}(T)=\sum_{j=1}^{2^n} X_j$.
  \end{algorithmic}
\end{algo}
The equivalence is in the sense that both processes are piecewise constant, c\`adl\`ag and all intermediate points follow the same conditional L\'evy bridge distributions. 
Note that $\ell^{(n)}$ coincides with $\overline\ell^{(n)}$ from Algorithm~\ref{algo:bridge} where the initial value has been chosen as $\mathcal X_{0,1}=\ell^{(n)}(T)=\sum_{j=1}^{2^n}X_j$.
The advantage of Algorithm~\ref{algo:bridge} is that $2^n$ independent samples from the same distribution have to be generated, instead of $2^n$ random variables from (different) conditional distributions.
As we will see in the following section, sampling from the distribution with density $f_{\gD_n}$ may be achieved if the characteristic function $\phi_\ell$ is available. 
In addition, we are still able to use the $L^p(\gO;\R)$ error bounds from Theorem~\ref{thm:bridge} for a fixed $n\in\N$.

\subsection{Inversion of the Characteristic Function}\label{ssec:Fourierinversion}
For an one-dimensional L\'evy process $\ell$ with characteristic function $\phi_\ell$, the characteristic function of any increment $\ell(t+\gD_n)-\ell(t)$ can be expressed via
\begin{equation*} 
\E[\exp(iu(\ell(t+\gD_n)-\ell(t)))]=\E[\exp(iu(\ell(\gD_n)))]=(\phi_\ell(u))^{\gD_n}
\end{equation*}
for any time step $\gD_n>0$. 
If $F_{\gD_n}$ denotes again the CDF of this increment, we obtain by Fourier inversion (see~\cite{GP51})
\begin{equation}\label{eq:finv}
F_{\gD_n}(x)=\frac{1}{2}-\int_\R \frac{(\phi_\ell(u))^{\gD_n}}{2\pi iu}\exp(-iux)du.
\end{equation} 
Using the well-known inverse transformation method (see also~\cite{AG07}) to sample from the CDF, allows us to reformulate Algorithm~\ref{alog:mod}:
\begin{algo}\label{algo:L_approx}
~
\begin{algorithmic}[1]
\State For $n\in\N$, fix $\gD_n,\Xi_n$ and generate i.i.d. $U_1,\dots,U_{2^n}$, where  $U_j\sim\cU([0,1])$ on $(\gO,\cA,\bP)$.
\State Determine $X_j:=\inf\{x\in\R|F_{\gD_n}(x)=U_j\}$ for $j=1,\dots,2^n$.  
\State Set $\ell^{(n)}(t)=0$ if $t\in[0,t_1)$, $\ell^{(n)}(t)=\sum_{k=1}^j X_k$ if $t\in[t_j,t_{j+1})$ for $j=1,\dots,2^n-1$ and $\ell^{(n)}(T)=\sum_{j=1}^{2^n}X_j$. 
\end{algorithmic}
\end{algo}
 
The evaluation of $F$ is crucial and may, in general, only be done numerically. To approximate the integral in Eq.~\eqref{eq:finv}, we employ the discrete Fourier inversion method introduced in~\cite{H98}. 
With this method the approximation error can be controlled with relatively weak assumptions on the characteristic function. 
Hence, the resulting algorithm is applicable for a broad class of L\'evy processes.
An alternative algorithm to approximate the CDFs of subordinating processes based on the inversion of Laplace transforms is described in~\cite{VT01}.
Although this approach seems promising in terms of computational effort, here we only consider the Fourier inversion technique. 
The latter is also applicable to L\'evy processes without bounded variation and yields uniform error bounds on the approximated CDF. 

\begin{assumption}\label{ass:1}
The distribution with characteristic function $(\phi_\ell)^{\gD_n}$ is continuous with finite variance and CDF $F_{\gD_n}$. Furthermore, 
\begin{itemize}
\item there exists a constant $R>0$ and $\eta>1$ such that $F_{\gD_n}(-x)\le R|x|^{-\eta}$ and $1-F_{\gD_n}(x)\le R|x|^{-\eta}$ for all $x>0$.
\item there exists a constant $B > 0$ and $\theta>0$ such that $|(\phi_\ell(u))^{\gD_n}|\le B|\frac{u}{2\pi}|^{-\theta}$ for all $u\in\R$. 
\end{itemize}
\end{assumption}
In case of infinite variance, we consider bounds on the density function instead:
\begin{assumption}\label{ass:2}
The distribution with characteristic function $(\phi_\ell)^{\gD_n}$ is continuous with density $f_{\gD_n}$. Furthermore,
\begin{itemize}
\item there exists a constant $R>0$ and $\eta>1$ such that $|f_{\gD_n}(x)|\le R|x|^{-\eta}$ for all $x\in\R$.
\item there exists a constant $B > 0$ and $\theta>0$ such that $|(\phi_\ell(u))^{\gD_n}|\le B|\frac{u}{2\pi}|^{-\theta}$ for all $u\in\R$. 
\end{itemize}
\end{assumption}

\begin{rem} \label{rem:2pi} 
In the case that $\theta>1$, we have that 
\begin{equation*}
 \int_\R|(\phi_\ell(u))^{\gD_n}|du\le 2+2B\int_1^\infty\left(\frac{u}{2\pi}\right)^{-\theta}<\infty,
\end{equation*}
which already implies the existence of a continuous density $f_{\gD_n}$ in both scenarios, see for example \cite[Proposition 28.1]{S99}.
Usually, $F_{\gD_n}$ or $f_{\gD_n}$ are unknown, but only the characteristic function~$(\phi_\ell)^{\gD_n}$ is given. 
To obtain $R$ and $\eta$, one can choose $R=(-1)^{k}\frac{d^{2k}}{du^{2k}}((\phi_\ell(u))^{\gD_n})\big|_{u=0}$ and $\eta=2k$ in Ass.~\ref{ass:1}, resp. $R=\frac{1}{2\pi}\int_\R|\frac{d^{k}}{du^{k}}((\phi_\ell(u))^{\gD_n})|du$ and $\eta=k$ in Ass.~\ref{ass:2}, where $k$ is any non-negative integer such that the derivatives exist (see~\cite[Lemma 12 and 13]{H98}). 
For example, in the first set of assumptions, the finite variance ensures that we can use $\eta=2$ and $R$ equal to the second moment of the distribution with characteristic function $(\phi_\ell)^{\gD_n}$.
\end{rem}
As an approximation of $F_{\gD_n}$ as in Eq.~\eqref{eq:finv} we introduce the function $\widetilde F_{\gD_n}$ given by
\begin{equation*}
\widetilde F_{\gD_n}(x):=\sum_{k=-M/2}^{M/2}q_k\exp(-i2\pi kx/J),
\end{equation*} 
for $x\in\R$, where 
\begin{equation*}
q_k:=
\begin{cases}
1/2&\text{for $k=0$}\\
\frac{1-\cos(2\pi\gk k)}{i2\pi k}(\phi_\ell(-2\pi k/J))^{\gD_n}&\text{for $0<|k|<M/2$}\\
0&\text{for $k=M/2$}\\
\end{cases},
\end{equation*} 
$M$ is an even integer and $\gk,J>0$ are parameters which are determined below. Note that $q_k=\overline{q_{-k}}$, where $\overline z$ denotes the complex conjugate for $z\in \mathbb C$.  
The Hermitean symmetry also holds for the function $k\mapsto\exp(-i2\pi kx/J)$. This ensures that, for every $x\in\R$, we have 
\begin{align*}
\widetilde F_{\gD_n}(x)
&=\frac{1}{2}+\sum_{k=1}^{M/2-1}q_k\exp(-i2\pi kx/J)+\overline{q_k\exp(-i2\pi kx/J)}\\
&=\frac{1}{2}+2\,\text{Re}\Big(\sum_{k=1}^{M/2-1}q_k\exp(-i2\pi kx/J)\Big)
\end{align*}
and hence $\widetilde F_{\gD_n}(x)\in\R$ for any real-valued argument $x$. The last identity should be exploited during the simulation to save computational time as here only half the summation is required.
Lastly, we denote by $\zeta(z,s):=\sum_{k=0}^\infty(k+s)^{-z}$ for $s,z\in\mathbb{C}$ with $\text{Re}(s)>0$ and $\text{Re}(z)>1$ the \textit{Hurwitz zeta function} and define as in~\cite{H98}
\begin{align*}
V_1(\gk,\eta)&:=(\gk/2)^{-\eta}+2\zeta(\eta,1-\frac{\gk}{2})+\zeta(\eta,1+\frac{\gk}{2})+\zeta(\eta,1-\frac{3\gk}{2}),\\
V_2(\gk,\eta)&:=\frac{2^{\eta-1}\gk^{1-\eta}}{\eta-1}+\frac{\gk}{2}\Big(2\zeta(\eta,1-\frac{\gk}{2})+\zeta(\eta,1+\frac{\gk}{2})+\zeta(\eta,1-\frac{3\gk}{2})\Big).
\end{align*}
The expressions $V_1(\gk,\eta)$ and $V_2(\gk,\eta)$ establish conditions on the choice of the (not yet determined) parameter $\gk$ in Theorem~\ref{thm:CDF_approx}. 
For a given domain parameter $D>0$ and accuracy $\eps>0$ the approximation $\widetilde F_{\gD_n}$ should fulfill the error bound
\begin{equation*} 
|\widetilde F_{\gD_n}(x)-F_{\gD_n}(x)|<\eps\quad\text{for $x\in[-D/2,D/2]$}.
\end{equation*}  
Once $\gk$ is determined, this can be achieved by choosing a sufficiently large parameter $J$ and, based on this $J$, a sufficiently large number of summands $M$. Admissible values for $\gk$, $J$ and $M$ depend on $D$, $\eps$ and the constants in Assumption~\ref{ass:1} resp.~\ref{ass:2}.  
\begin{thm}(\cite[Theorem 10 and 11]{H98}) \label{thm:CDF_approx}
Let $D>0$ and $\eps>0$.
If Assumption~\ref{ass:1} holds, choose $\gk$, $J$ and $M$ such that 
\begin{equation*}
0<\gk<\frac{2}{3}\quad\text{and}\quad\gk^\eta V_1(\gk,\eta)\le2^{\eta+1},
\end{equation*}
\begin{equation*}
J\ge\frac{D}{\gk}\quad\text{and}\quad J\ge\Big(\frac{3RV_1(\gk,\eta)}{2\eps}\Big)^{1/\eta},
\end{equation*}
and
\begin{equation*}
M\ge2+2J\Big(\frac{6B}{\eps\pi\theta}\Big)^{1/\theta}.
\end{equation*}
If Assumption~\ref{ass:2} holds, choose $\gk$, $J$ and $M$ such that 
\begin{equation*}
0<\gk<\frac{2}{3}\quad\text{and}\quad\gk^{\eta-1} V_2(\gk,\eta)\le\frac{2^\eta}{\eta-1},
\end{equation*}
\begin{equation*}
J\ge\frac{D}{\gk}\quad\text{and}\quad J\ge\Big(\frac{3RV_2(\gk,\eta)}{2\eps}\Big)^{1/(\eta-1)},
\end{equation*}
and
\begin{equation*}
M\ge2+2J\Big(\frac{6B}{\eps\pi\theta}\Big)^{1/\theta}.
\end{equation*}
This yields, for either case, that $|F_{\gD_n}(x)-\widetilde F_{\gD_n}(x)|<\eps$ for all $x\in[-D/2,D/2]$ and it is always possible to find a $\gk$ that meets the given conditions. 
\end{thm}

\begin{rem}\label{rem:Deps}
In \cite{H98}, by
\begin{equation*}
J\ge\frac{2}{\gk}\left(\frac{3R}{\eps}\right)^{1/\eta}\text{resp.}\quad J\ge\frac{2}{\gk}\left(\frac{3R}{\eps(\eta-1)}\right)^{1/(\eta-1)}
\end{equation*}
in fact stricter conditions are imposed on $J$. 
The proofs of Theorems 10 and 11 in~\cite{H98} still give immediately a proof for Theorem~\ref{thm:CDF_approx}. 
The advantage of the bounds in Theorem~\ref{thm:CDF_approx} is that they produce a smaller approximation error in the following analysis (see also Remark~\ref{rem:Deps2}).
We refer to~\cite{H98} for an optimal choice of $\gk$ depending on $\eta$. 
Once $\gk$ is determined, it is favorable to choose $D$ and $\eps$ in a way such that none of the parameters has a dominant effect on the resulting number of summations $M$. 
This is ensured if the two lower bounds on $J$ are equal, meaning for fixed $D>0$ we set
\begin{equation*}
\eps=\frac{3}{2}RV_1(\gk,\eta)\gk^\eta D^{-\eta}\quad\text{resp.}\quad \eps=\frac{3}{2}RV_2(\gk,\eta)\gk^{\eta-1}D^{-\eta+1}
\end{equation*}
if the first resp. second set of assumptions holds.
\end{rem}

Since the approximation error $|\widetilde F_{\gD_n}(x)-F_{\gD_n}(x) |$ is only bounded for $x\in[-D/2,D/2]$, we have to modify the third step in Algorithm~\ref{algo:L_approx}:
\begin{algo}\label{algo:approx2}
~
\begin{algorithmic}[1]
\State For $n\in\N$, fix $\gD_n,\Xi_n$ and generate i.i.d. $(U_j\sim\cU([0,1]),j=1,\ldots,2^n)$ on $(\gO,\cA,\bP)$.
\State Set, for $j=1,\dots,2^n$
\begin{equation*}
\widetilde X_j=
\begin{cases}
-D/2 & \text{if $U_j<\min\{ \widetilde F_{\gD_n}([-D/2,D/2])\}$}\\
D/2 & \text{if $U_j>\max\{ \widetilde F_{\gD_n}([-D/2,D/2])\}$}\\
\inf\{x\in[-D/2,D/2]\big|\widetilde F_{\gD_n}(x)=U_j \} &\text{if $U_j\in\widetilde F_{\gD_n}([-D/2,D/2])$}
\end{cases}.
\end{equation*}
\State Set $\widetilde\ell^{(n)}(t)=0$ if $t\in[0,t_1)$, $\widetilde\ell^{(n)}(t)=\sum_{k=1}^j\widetilde X_k$ if $t\in[t_j,t_{j+1})$ for $j=1,\dots,2^n-1$ and $\widetilde\ell^{(n)}(T)=\sum_{j=1}^{2^n} \widetilde X_j$. 
\end{algorithmic}
\end{algo}

Intuitively, if we choose $D$ large and $\eps$ small enough, the atoms in the distribution of $\widetilde X_i$ at $\pm D/2$ disappear. 
The function $\widetilde F_{\gD_n}$ is then sufficiently close to the CDF $F_{\gD_n}$, hence the generated random variables $\widetilde X_i$ will have a distribution similar to $F_{\gD_n}$. 
From here on, we define $X$ as the random variable which is generated from $U\sim\mathcal{U}([0,1])$ by inversion of the (exact) CDF $F_{\gD_n}$ and $\widetilde X$ as the random variable generated from $U$ by inversion of the approximated CDF $\widetilde F_{\gD_n}$.

\begin{thm} \label{thm:conv_in_dist}
Let $\widetilde F_{\gD_n}$ be the approximation of $F_{\gD_n}$ which is valid for parameters $D>0$ and $\eps>0$ in the sense of Theorem~ \ref{thm:CDF_approx}. 
Then $\widetilde X$ converges in distribution to a random variable $X$ with CDF equal to $F_{\gD_n}$ as $D\to\infty$ and $\eps\to0$.
\end{thm} 
\begin{proof} First, note that $\widetilde F_{\gD_n}$ is not necessarily monotone and might admit arbitrary values outside of $[-D/2,D/2]$, thus cannot be regarded as a CDF. Since $\widetilde X$ only admits values in the desired interval, we obtain probability zero for the event that $|\widetilde X|>D/2$. With this in mind we construct the CDF of $\widetilde X$ and show its convergence in distribution using Portmanteau's theorem. We define the function 
\begin{equation*}
\widehat F:\R\to [0,1],\quad x\mapsto
\begin{cases}
0 & \text{if $x<-D/2$}\\
\min(1,m_D(x))\mathbf1_{\{m_D(x)>0\}}  & \text{if $x\in[-D/2,D/2]$}\\
1 & \text{if $x>D/2$}\\
\end{cases},
\end{equation*}
where $m_D(x):=\max_{y\in[-D/2,x]}\widetilde F_{\gD_n}(y)$. The continuity of $\widetilde F_{\gD_n}$ guarantees that $m_D(x)$ is well-defined for each $x\in[-D/2,D/2]$. Clearly, $\widehat F$ is monotone increasing and 
$$\mathbb{P}(\widetilde X\le x)=\widehat F(x)$$ if $|x|>D/2$. For $|x|\le D/2$ we have that 
\begin{equation*}
\begin{split}
\mathbb{P}(\widetilde X\le x)&=\mathbb{P}(\inf\{|y|\le D/2\,|\,\widetilde F_{\gD_n}(y)\ge U\}\le x)\\
&=\mathbb{P}(\max\limits_{y\in[-D/2,x]}\widetilde F_{\gD_n}(y)\ge U)\\
&=\min(1,m_D(x))\mathbf1_{\{m_D(x)>0\}}=\widehat F(x), 
\end{split}
\end{equation*}
hence $\widehat F$ is the CDF of $\widetilde X$. 
With the monotonicity of $F_{\gD_n}$ and $|F_{\gD_n}-\widetilde F_{\gD_n}|<\eps$ on $[-D/2,D/2]$ we get
\begin{align*}
\widehat F(x)&=\min(1,m_D(x))\mathbf1_{\{m_D(x)>0\}}\le\min(1,\max\limits_{y\in[-D/2,x]}\widetilde F_{\gD_n}(y))\\
&\le\min(1,\max\limits_{y\in[-D/2,x]}F_{\gD_n}(y)+\eps)=\min(1,F_{\gD_n}(x)+\eps),
\end{align*}
for $x\in[-D/2,D/2]$ and analogously 
\begin{equation*}
\widehat F(x)\ge\min(1,\max\limits_{y\in[-D/2,x]}F_{\gD_n}(y)-\eps)\mathbf1_{\{\max\limits_{y\in[-D/2,x]}F_{\gD_n}(y)-\eps>0\}}=\max(F_{\gD_n}(x)-\eps,0),
\end{equation*}
thus 
\begin{equation*}
|\widehat F(x)-F_{\gD_n}(x)|\le\eps.
\end{equation*}
We choose sequences $(D_k,k\in\N)$ and $(\eps_m,m\in\N)$ with $\lim_{k\to\infty}D_k=+\infty$, $\lim_{m\to\infty}\eps_m=0$ and denote by $\widehat F_{k,m}$ the CDF of the random variables $\widetilde X_{k,m}$ corresponding to each $D_k$ and $\eps_m$. 
For every $x\in\R$ there is some $k^*$ such that $x\in[-D_k/2,D_k/2]$ for all $k\ge k^*$, hence
$$\lim\limits_{m\to\infty}\lim\limits_{k\to\infty}|\widehat F_{k,m}(x)-F_{\gD_n}(x)|\le\lim\limits_{m\to\infty}\eps_m=0$$
and the claim follows by Portmanteau's theorem.
\end{proof}

\begin{rem}
Before showing the convergence of $\widetilde X$ to $X$ in $L^p(\gO;\R)$, we have to make sure that the random variables $\widetilde X$ generated by Algorithm~\ref{algo:approx2} are actually defined on the same probability space $(\gO,(\cA_t,t\geq 0),\bP)$ as $X$.
Since $X$ represents the increment of a L\'evy process $\ell$ on $(\gO,(\cA_t,t\geq 0),\bP)$ with CDF $F_{\gD_n}$, we may define the mapping $U:=F_{ \gD_n}\circ X:\gO\to[0,1]$.
It is then easily verified that $U$ is a  $\cU([0,1])$-distributed random variable. 
For fixed parameters $D,\eps>0$ and an approximation $\widetilde F_{\gD_n}$ of $F_{\gD_n}$ we define the \textit{pseudo inverse} of $\widetilde F_{\gD_n}$ (as in  Algorithm~\ref{algo:approx2}) as
\begin{equation*}
\widetilde F_{\gD_n}^{-1}:[0,1]\to\R,\;\, u\mapsto
\begin{cases}
-D/2 & \text{if $u<\min\{ \widetilde F_{\gD_n}([-D/2,D/2])\}$}\\
D/2 & \text{if $u>\max\{ \widetilde F_{\gD_n}([-D/2,D/2])\}$}\\
\inf\{x\in[-D/2,D/2]\big|\widetilde F_{\gD_n}(x)=u\} &\text{if $u\in\widetilde F_{\gD_n}([-D/2,D/2])$}
\end{cases}.
\end{equation*}
We note that $\widetilde F_{\gD_n}^{-1}$ is a piecewise continuous, thus measurable, mapping which implies that $\widetilde X=\widetilde F_{\gD_n}^{-1}\circ F_{\gD_n}\circ X:\gO\to[-D/2,D/2]$ is a random variable on $(\gO,(\cA_t,t\geq 0),\bP)$.
\end{rem}

Under additional, but natural, assumptions, it is possible to show stronger convergence results of the approximation for both sets of assumptions.
\begin{thm}($L^p(\gO;\R)$-convergence I)\label{thm:L^p 1}
Let $F_{\gD_n}$ be continuously differentiable on $\R$ with density $f_{\gD_n}$ (see Remark~\ref{rem:2pi}) and $(\phi_\ell)^{\gD_n}$ be bounded as in Assumption \ref{ass:1} with $\eta>1$. 
Furthermore, assume that the approximation parameters $D$ and $\eps$ fulfill $D=C\eps^{-d}$ for $C,d>0$. 
If $d<\frac{1}{p}$, then for all $p\in[1,\eta)$
\begin{equation*}
\E(|\widetilde X-X|^p)\to0 \quad\text{as $\eps\to 0$}.
\end{equation*}

\end{thm}
\begin{proof}
Let $\eps>0$, $D=C\eps^{-d}$ and $p\in[1,\eta)$ be as in the claim. We split the expectation in the following way
\begin{equation*}
\E(|\widetilde X-X|^p)=\E(|\widetilde X-X|^p\mathbf{1}_{\{|X|>D/2\}})+\E(|\widetilde X-X|^p\mathbf{1}_{\{|X|\le D/2\}}),
\end{equation*}
and show the convergence for each term on the right hand side. 
Recall that $\widetilde X\in[-D/2,D/2]$ by construction. 
We obtain for the first term 
\begin{align*}
\E(|\widetilde X-X|^p\mathbf{1}_{\{|X|>D/2\}})&\le\int_{D/2}^\infty|-D/2-x|^pf_{\gD_n}(x)dx+\int_{-\infty}^{-D/2}|D/2-x|^pf_{\gD_n}(x)dx\\
&=\int_{D/2}^\infty(D/2+x)^p(f_{\gD_n}(x)+f_{\gD_n}(-x))dx\\
&=\int_{D/2}^\infty\int_{0}^{D/2+x}py^{p-1}dy(f_{\gD_n}(x)+f_{\gD_n}(-x))dx.
\end{align*}
Using that 
$$(x,y)\in(D/2,\infty)\times(0,D/2+x) \Leftrightarrow (x,y)\in (D/2,\infty)\times(0,D/2)\cup(y,\infty)\times(D/2,\infty),$$ 
we may use Fubini's theorem to exchange the order of integration and rewrite 
\begin{align*}
\E(|\widetilde X-X|^p\mathbf{1}_{\{|X|>D/2\}})&\le\int_{D/2}^\infty\int_{0}^{D/2+x}py^{p-1}dy(f_{\gD_n}(x)+f_{\gD_n}(-x))dx\\
&=\int_0^{D/2} \int_{D/2}^\infty(f_{\gD_n}(x)+f_{\gD_n}(-x))dx py^{p-1}dy \\
&\quad+ \int_{D/2}^\infty \int_{y}^{\infty} (f_{\gD_n}(x)+f_{\gD_n}(-x))dx py^{p-1}dy\\
&=\int_0^{D/2} (1-F_{\gD_n}(D/2)+F_{\gD_n}(-D/2)) py^{p-1}dy \\
&\quad+\int_{D/2}^\infty (1-F_{\gD_n}(y)+F_{\gD_n}(-y)) py^{p-1}dy .\\
\end{align*}
With the bounds on $F_{\gD_n}$ from Assumption~\ref{ass:1} we then have 
\begin{align*}
\E(|\widetilde X-X|^p\mathbf{1}_{\{|X|>D/2\}})&\le 2R (D/2)^{-\eta} \int_0^{D/2} py^{p-1}dy+2Rp\int_{D/2}^\infty \frac{y^{p-1}}{y^\eta}dy \\
&= 2R (D/2)^{p-\eta}+ 2Rp\zeta (\eta+1-p,D/2).
\end{align*}
Note that the Hurwitz zeta function $\zeta$ is well-defined (as $\eta>p$) and converges to 0 as $D\to\infty$.

For the second term, consider two realizations of the random variables $X(\go)$ and $\widetilde X(\go)$ for some $\go\in\gO$, where $|X(\go)|\le D/2$. $F_{\gD_n}$ is continuously differentiable by assumption, hence 
$$F_{\gD_n}(X(\go))-F_{\gD_n}(\widetilde X(\go))=f_{\gD_n}(\xi(\go))(X(\go)-\widetilde X(\go)),$$
with $\xi(\go)$ lying in between $X(\go)$ and $\widetilde X(\go)$, meaning $|\xi(\go)|\le D/2$ and 
\begin{equation*}
\1_{\{|X(\go)|\le D/2\}}(\go)\le\1_{\{|\xi(\go)|\le D/2\}}(\go).
\end{equation*}
For $\widetilde\eps:=C^{-1}\eps^{d+1}>0$ we split the expectation once more into
\begin{align*}
\E(|\widetilde X-X|^p\mathbf{1}_{\{|X|\le D/2\}})
\le&\E(|\widetilde X-X|^p\1_{\{|\xi|\le D/2\}})\\
\le&\underbrace{\E(|\widetilde X-X|^p\1_{\{|\xi|\le D/2,f(\xi)\ge\widetilde\eps\}})}_{:=I}+\underbrace{\E(|\widetilde X-X|^p\1_{\{|\xi|\le D/2,f(\xi)<\widetilde\eps\}})}_{:=II}.
\end{align*}
In case that $f_{\gD_n}(\xi(\go))\ge\widetilde\eps$, we can rearrange the terms to 
$$|\widetilde X(\go)-X(\go)|^p=\frac{|F_{\gD_n}(\widetilde X(\go))-F_{\gD_n}(X(\go))|^p}{f_{\gD_n}(\xi(\go))^p}.$$
If $X$ and $\widetilde X$ are generated by $U\sim\mathcal U([0,1])$ and $\widehat F_{\gD_n}$ denotes again the CDF of $\widetilde X$, this yields
\begin{align*} 
|\widetilde X(\go)-X(\go)|^p&=\frac{|F_{\gD_n}(\widetilde X(\go))-\widehat F_{\gD_n}(\widetilde X(\go))|^p}{f_{\gD_n}(\xi(\go))^p}
<\frac{\eps^p}{f_{\gD_n}(\xi(\go))^p},
\end{align*}
where we have used that $U(\go)=F_{\gD_n}(X(\go))=\widehat F_{\gD_n}(\widetilde X(\go))$ and $|F_{\gD_n}(\widetilde X(\go))-\widehat F_{\gD_n}(\widetilde X(\go))|<\eps$ (see Theorem~\ref{thm:conv_in_dist}) in the second step. 
This gives a bound for $I$:    
\begin{equation} \label{III}
\begin{split}
I&< \eps^p\,\E(f_{\gD_n}(\xi)^{-p}\1_{\{|\xi|\le D/2,f_{\gD_n}(\xi)\ge\widetilde\eps\}})\\
&=\eps^p\int_{-D/2}^{D/2}\1_{\{f_{\gD_n}(\xi)\ge\widetilde\eps\}}f_{\gD_n}(\xi)^{1-p}d\xi\le\frac{\eps^p}{\widetilde\eps^{p-1}}\int_{-D/2}^{D/2}\1_{\{f_{\gD_n}(\xi)\ge\widetilde\eps\}}d\xi.
\end{split}
\end{equation}
If $f_{\gD_n}(\xi(\go))<\widetilde\eps$, we obtain by $|\widetilde X(\go)-X(\go)|\mathbf{1}_{\{|X(\go)|\le D/2\}}\le D$
\begin{equation} \label{IV}
\begin{split}
II&\le D^p\,\E(\1_{\{|\xi|\le D/2,f_{\gD_n}(\xi)<\widetilde\eps\}})\\
&=D^p\int_{-D/2}^{D/2}\1_{\{f_{\gD_n}(\xi)<\widetilde\eps\}}f_{\gD_n}(\xi)d\xi<D^p\,\widetilde\eps\int_{-D/2}^{D/2}\1_{\{f_{\gD_n}(\xi)<\widetilde\eps\}}d\xi
\end{split}
\end{equation}
and hence by Eqs.~\eqref{III} and \eqref{IV} and $\widetilde\eps=C^{-1}\eps^{1+d}$
\begin{equation*}
\E(|\widetilde X-X|^p\mathbf{1}_{\{|X|\le D/2\}})\le I+II< D^p\,\widetilde\eps\Big(\int_{-D/2}^{D/2}\1_{\{f_{\gD_n}(\xi)\ge\widetilde\eps\}}d\xi+\int_{-D/2}^{D/2}\1_{\{f_{\gD_n}(\xi)<\widetilde\eps\}}d\xi\Big)=D^{p+1}\widetilde\eps.
\end{equation*}
With the estimate for $\E(|\widetilde X-X|^p\mathbf{1}_{\{|X|> D/2\}})$, $D=C\eps^{-d}$ and $\widetilde\eps=C^{-1}\eps^{1+d}$ this leads to 
\begin{align*}
\E(|\widetilde X-X|^p)&\le 2Rp\zeta(\eta+1-p,D/2)+2R(D/2)^{p-\eta}+D^{p+1}\widetilde\eps\\ 
&=2Rp\zeta(\eta+1-p,D/2)+2R(D/2)^{p-\eta}+C^{1/d} D^{p-1/d},
\end{align*}
and since $0<d<\frac{1}{p}$ and $\eta>p$ by assumption,
\begin{equation*}
\E(|\widetilde X-X|^p)\to0 \quad \text{as $\eps\to0$}.
\end{equation*}

\end{proof}
\begin{rem}
The relation $\widetilde\eps=C^{-1}\eps^{1+d}$ is chosen such that the factors preceding the integrals in Eqs.~\eqref{III} and~\eqref{IV} are equilibrated.
As only the sum of the two integrals is known a-priori, this leads to a better error estimation compared to non-equilibrated factors.
\end{rem}

\begin{thm}($L^p(\gO;\R)$-convergence II)\label{thm:L^p 2}
Let $F_{\gD_n}$ be continuously differentiable on $\R$ with density $f_{\gD_n}$ and $(\phi_\ell)^{\gD_n}$ be bounded as in Assumption \ref{ass:2} with $\eta>2$. 
Furthermore, assume that the approximation parameters $D$ and $\eps$ fulfill $D=C\eps^{-d}$ for $C,d>0$. 
If $d<\frac{1}{p}$, then for all $p\in[1,\eta)$
\begin{equation*}
\E(|\widetilde X-X|^p)\to0 \quad\text{as $\eps\to 0$}.
\end{equation*} 
\end{thm}
\begin{proof}
Let $\eps>0$, $D=C\eps^{-d}$ and $p\in[1,\eta-1)$. Again, we split the expectation into
\begin{equation*}
\E(|\widetilde X-X|^p)=\E(|\widetilde X-X|^p\mathbf{1}_{\{|X|>D/2\}})+\E(|\widetilde X-X|^p\mathbf{1}_{\{|X|\le D/2\}}),
\end{equation*}
and show convergence for the first term only, as the second term can be treated analogously to Theorem~\ref{thm:L^p 1}.
In the same way as in Theorem~\ref{thm:L^p 1}, we may write for the first term
\begin{align*}
\E(|\widetilde X-X|^p\mathbf{1}_{\{|X|>D/2\}})\le\int_{D/2}^\infty(D/2+x)^p(f_{\gD_n}(x)+f_{\gD_n}(-x))dx,
\end{align*}
and further, by Assumption~\ref{ass:2}, follows
\begin{align*}
\int_{D/2}^\infty(D/2+x)^p(f_{\gD_n}(x)+f_{\gD_n}(-x))dx &\le 2R\int_{D/2}^\infty(D/2+x)^px^{-\eta}dx \\
&=2^{p+1}R\int_0^{\infty}\frac{(D/2+x/2)^p}{(D/2+x)^\eta}dx\\
&<2^{p+1}R\zeta(\eta-p,C\eps^{-d}/2),
\end{align*}
which tends to zero as $\eps\to 0$, because $\eta>p+1$.\\
\end{proof}

\begin{rem} \label{rem:Deps2}
As expected, the admissible range of values for $d$ and $\eta$ narrows as the rate of convergence $p$ increases.
For example, to obtain $L^2(\gO;\R)$-convergence, we need $d<\frac{1}{2}$ and $\eta>2$ in Theorem~\ref{thm:L^p 1} and $\eta>3$ in Theorem~\ref{thm:L^p 2}. 
Recall Remark~\ref{rem:Deps}, where we have concluded that optimal relations between $D$ and $\eps$ are given by
\begin{align*}
&D=\gk\left(3/2\,RV_1(\gk,\eta)\right)^{1/\eta}\eps^{-1/\eta} \quad &&\text{if Assumption~\ref{ass:1} holds and}\\
&D=\gk\left(3/2\,RV_2(\gk,\eta)\right)^{1/(\eta-1)}\eps^{-1/(\eta-1)}\quad &&\text{if Assumption~\ref{ass:2} holds}.
\end{align*}
For $L^p(\gO;\R)$-convergence, we need in both cases $D=C\eps^{-d}$, where $C>0$ and $d\in(0,1/p)$. 
Hence, we can simply use $C=\gk(3/2\,RV_1(\gk,\eta))^{1/\eta}$ and $d=1/\eta<1/p$ in the first scenario and $C=\gk(3/2\,RV_2(\gk,\eta))^{1/\eta}$ and $d=1/(\eta-1)<1/p$ for the second set of assumptions.
This explains the bounds on $J$ (see also Remark~\ref{rem:Deps}):
In Theorem~\ref{thm:L^p 1}, we obtain the expression $C^\eta D^{p-\eta}$ as a term of the overall error. 
If we had used the restrictions on $J$ as in \cite{H98}, we would have used $C=2(3R)^{1/\eta}§$ instead of the choice above and this would have resulted in an error term $C^\eta D^{p-\eta}$ being nearly twice as large (the argumentation works analogously for Theorem~\ref{thm:L^p 2}).
\end{rem}

\begin{ex}\label{ex:eta_cond}
The conditions $\eta>p$ in Theorem~\ref{thm:L^p 1} and $\eta>p-1$ in Theorem~\ref{thm:L^p 2} can not be relaxed as the following examples show:
First, we investigate the Student's t-distribution with 3 degrees of freedom and density function
\begin{equation*}
f^{t3}(x)=\frac{\gG_G(2)}{\sqrt{3\pi}\gG_G(3/2)}\left(1+\frac{x^2}{3}\right)^{-2},
\end{equation*}
where $x\in\R$ and $\gG_G(\cdot)$ is the Gamma function.
As shown in \cite{K71}, this distribution is infinitely divisible and has characteristic function
\begin{equation*}
\phi_{t3}(u):=\exp(-\sqrt3 |u|)(\sqrt3 |u|+1),
\end{equation*}
hence we can define a L\'evy process $(\ell^{t3}(t),t\in\T)$ with $\phi_{t3}(u)$ as characteristic function.
For simplicity we set $\gD_n=1$. 
In this case the (symmetric) distribution of the increment $\ell^{t3}(t+\gD_n)-\ell^{t3}(t)$ has zero mean, finite variance, and its CDF $F_1$ 
can be bounded for $x>0$~by 
\begin{equation*}
F_1(-x)=1-F_1(x)=\int_{-\infty}^{-x}f^{t3}(y)dy<\frac{\gG_G(2)}{\sqrt{3\pi}\gG_G(3/2)}\int_{-\infty}^{-x}\frac{3^2}{y^4}dy=\frac{\sqrt3\gG_G(2)}{\sqrt{\pi}\gG_G(3/2)}x^{-3} =:Rx^{-3}.
\end{equation*}
Thus, this yields $\eta=3$. The bounds for $\phi_{t3}$ are also straightforward:
\begin{equation*}
|\phi_{t3}(u)|\le(2\pi)^{-1}\max\limits_{\widehat u>0}\exp(-\sqrt3 \widehat u)(\sqrt3 \widehat u^2+\widehat u) |\frac{u}{2\pi}|^{-1} =:B |\frac{u}{2\pi}|^{-1},
\end{equation*}
where the maximum in $B$ is found by differentiation giving $\widehat u=(1+\sqrt5)/(2\sqrt3)$.
Now, all requirements for $L^3$-convergence except $\eta>3$ are fulfilled.
But the t-distribution with 3 degrees of freedom does not admit a third moment, hence we cannot have $L^3(\gO;\R)$-convergence although $\eta=3$.

For the second case we consider the (standard) Cauchy process with characteristic function $(\phi_C(u))^t=\exp(-t|u|)$. 
It can be shown that the increment over time $\gD_n>0$ is again Cauchy-distributed with density 
$$f^C_{\gD_n}(x)=\frac{\gD_n}{\pi(\gD_n^2+x^2)}.$$
This means the CDF of the increment is continuously differentiable and the bounds as in Assumption~\ref{ass:2} are easily found by $f^C_{\gD_n}(x)\le (\gD_n/\pi)|x|^{-2}$ for $x\in\R$ and 
\begin{equation*}
 |(\phi_C(u))^{\gD_n}|\le (2\pi)^{-1}\max\limits_{u\in\R}u\exp(-\gD_n |u|)|\frac{u}{2\pi}|^{-1}=(2\pi\gD_n)^{-1}\exp(-1)|\frac{u}{2\pi}|^{-1}
\end{equation*}
for $u\in\R$. But clearly, $L^p(\gO;\R)$-convergence in the sense of Theorem~\ref{thm:L^p 2} for any $p\ge1$ is impossible, as the Cauchy process does not have any finite moments. 
\end{ex}

From $L^p$-convergence follows almost sure convergence by a Borel--Cantelli-type argument, given $\eta$ in Assumptions~\ref{ass:1} and~\ref{ass:2} is large enough.

\begin{cor}\label{cor:P-as}
Under the assumptions of Theorem~\ref{thm:L^p 1}, set $\psi_1:=\min\left(d\eta,1-d\right)$, let $m\in\N$ and set $\eps=\eps_m=m^{-q}$, with $q>\psi_1^{-1}$. 
If $(\widetilde X_m,m\in\N)$ is generated based on the sequence $(\eps_m,m\in\N)$ (and the corresponding $D_m=C\eps_m^{-d}$), then $(\widetilde X_m,m\in\N)$ converges to $X$ 
$\bP$-almost surely.
\end{cor}
\begin{proof}
If the assumptions of Theorem~\ref{thm:L^p 1} with $\eta>1$ hold, we can ensure at least $L^1(\gO;\R)$ convergence. 
Note that the Hurwitz zeta function $\zeta(\eta+1-p,D/2)=\zeta(\eta+1-p,C/2\eps^{-d})$ is of order $\mathcal{O}(\eps^{d(\eta+1-p)})$ as $\eps\to0$.
With Markov's inequality, $p=1$ and the given error bounds, we get that for each $\widehat{\eps}>0$ and $m\in\N$
\begin{equation*}
\mathbb{P}(|\widetilde X_m-X|>\widehat\eps\,)\le\frac{\E(|\widetilde X_m-X|)}{\widehat\eps}
\le \frac{\widetilde C }{\widehat\eps}\left(\eps_m^{d\eta}+\eps_m^{1-d}\right)
\le \frac{2\widetilde C}{\widehat\eps}\eps_m^{\psi_1},
\end{equation*}
(recall that $1>\psi_1>0$ and $\eps_m\le1$) where the constant $\widetilde C>0$ depends on $R,\eta$ and $C$. But this means 
\begin{equation*}
\sum_{m=1}^\infty\mathbb{P}(|\widetilde X_m-X|>\widehat\eps)
\le \frac{2\widetilde C}{\widehat\eps}\sum_{m=1}^\infty\eps_m^{\psi_1}
= \frac{2\widetilde C}{\widehat\eps}\sum_{m=1}^\infty m^{-q\psi_1}<\infty,
\end{equation*}
since $q\psi_1>1$ by construction. The claim then follows by the Borel-Cantelli lemma.
\end{proof}
\begin{cor}
Under the assumptions of Theorem~\ref{thm:L^p 2}, set $\psi_2:=\min\left(d(\eta-1),1-d\right)$, let $m\in\N$ and set $\eps=\eps_m=m^{-q}$, with $q>\psi_2^{-1}$. 
If $(\widetilde X_m,m\in\N)$ is generated based on the sequence $(\eps_m,m\in\N)$ (and the corresponding $D_m=C\eps_m^{-d}$), then $(\widetilde X_m,m\in\N)$ converges to $X$ 
$\bP$-almost surely.
\end{cor}

We can now combine the error estimates for any increment over time $\gD_n>0$ with the piecewise approximation error from Algorithm~\ref{algo:bridge} to bound the overall error $\ell(t)-\widetilde\ell^{(n)}(t)$.
\begin{thm}\label{thm:L^p}
Let $\ell$ be a L\'evy process on $(\gO,(\cA_t,t\geq 0),\bP)$ with characteristic function $\phi_\ell$, CDF  $F_{t}$ and density $f_{t}$ for any $t\in\T$.
Assume for $n\in\N$ and fixed $\gD_n$ there are constants $R,\eta,B,\theta>0$ such that either Ass~\ref{ass:1} or Ass.~\ref{ass:2} holds.
Let $\widetilde\ell^{(n)}$ be the piecewise constant approximation of $\ell$ generated by Algorithm~\ref{algo:approx2} and the approximation $\widetilde F_{\gD_n}$ of $F_{\gD_n}$. 
There are parameters $D_n,\eps_n$ for $\widetilde F_{\gD_n}$ such that for any $p\in[1,\eta)$ resp. $p\in[1,\eta-1)$ the approximation error is bounded by
$$\E(|\ell(t)-\widetilde\ell^{(n)}(t)|^p)^{1/p}\le C_{\ell,T,p,R,\eta}\gD_n^{1/p},\quad t\in\T,$$
where the constant $C_{\ell,T,p,R,\eta}>0$ only depends on the indicated parameters. 
\end{thm}
\begin{proof}
By Theorem~\ref{thm:bridge} we may regard $\ell$ as the (pointwise) $L^p(\gO;\R)$-limit process of the sequence $(\overline\ell^{(n)},n\in\N)$ generated by Algorithm~\ref{algo:bridge}.
For fixed $n$, we may then identify $\overline\ell^{(n)}$ with $\ell^{(n)}$ from Algorithm~\ref{algo:L_approx} to obtain
\begin{equation*}
 \E(|\ell(t)-\widetilde\ell^{(n)}(t)|^p)^{1/p}\le \E(|\ell(t)-\overline\ell^{(n)}(t)|^p)^{1/p}+ \E(|\ell^{(n)}(t)-\widetilde\ell^{(n)}(t)|^p)^{1/p}.
\end{equation*}
The first term is bounded by 
\begin{equation*}
\E(|\ell(t)-\overline\ell^{(n)}(t)|^p)^{1/p}\le C^{1/p}_{\ell,T,p}\sum_{i=n+1}^\infty2^{-i/p}=\frac{C^{1/p}_{\ell,T,p}2^{-n/p}}{2^{1/p}-1}.
\end{equation*}
For the treatment of the second term, we define for $j=1,\dots,2^n$ the random variables $X_j:=F_{\gD_n}^{-1}(U_j)\stackrel{\mathcal L}{=}\ell(\gD_n)$.
Here, $U_1,\dots,U_{2^n}$ is the i.i.d sequence of $\cU([0,1])$ random variables on $(\gO,(\cA_t,t\geq 0),\bP)$ from Algorithm~\ref{algo:L_approx}. 
The increments of the approximation $\widetilde\ell^{(n)}$ are then given by $\widetilde X_j:=\widetilde F^{(-1)}_{\gD_n}(U_j)$ which yields
\begin{align*}
|\ell^{(n)}(t)-\widetilde\ell^{(n)}(t)|\le\sum_{j=1}^{2^n}|X_j-\widetilde X_j|.
\end{align*}
The differences $(X_j-\widetilde X_j, j=1,\dots,2^n)$ are i.i.d. by construction, hence
\begin{align*}
\E\left(|\ell^{(n)}(t)-\widetilde\ell^{(n)}(t)|^p\right)^{1/p}
\le\sum_{j=1}^{2^n}\E(|X_j-\widetilde X_j|^p)^{1/p}
\le2^{n}\E(|X_1-\widetilde X_1|^p)^{1/p}.
\end{align*}
Now let $\widetilde F_{\gD_n}$ be the approximation of $F_{\gD_n}$ for some $\eps\in(0,1]$ and $D=C\eps^{-d}$. 
For the first set of assumptions, we apply the error estimates of Theorem~\ref{thm:L^p 1} to obtain
\begin{align*}
\E\left(|\ell^{(n)}(t)-\widetilde\ell^{(n)}(t)|^p\right)^{1/p}
&\le 2^{n}\E(|X_1-\widetilde X_1|^p)^{1/p}\\
&\le 2^{n} \big(2Rp\zeta(\eta+1-p,C\eps^{-d}/2)+2R(C\eps^{-d}/2)^{p-1/d}+C^p\eps^{1-dp}\big)^{1/p}\\
&\le 2^{n} C_{R,\eta,p} \eps^{\psi(p)/p},\\
\end{align*}
where $\psi(p):=\min\left(d(\eta+1-p),1-dp\right)$ and $C_{R,\eta,p}>0$ depends on $C$ and the indicated parameters.
An error of order $\gD_n^{1/p}$ in the last inequality is then achieved by choosing $\eps=\eps_n=2^{-(np+n)/\psi(p)}$ and $D_n=C\eps_n^{-d}$.
The proof for the second set of assumptions is carried out identically with the only difference that $\psi(p):=\min\left(d(\eta-p),1-dp\right)$.
\end{proof}
\begin{rem}\label{rem:t-conv}
For an efficient simulation one would choose $R$ based on $(\phi_\ell)^{\gD_n}$ as in Remark~\ref{rem:2pi} and then $C$ based on $R$ as suggested in Remark~\ref{rem:Deps2}. 
Note that in this case $R=\mathcal{O}(\gD_n)$ and 
\begin{equation*}
C=\gk\left(3/2\,RV_1(\gk,\eta)\right)^{1/\eta}=\mathcal{O}(\gD_n^{1/\eta})\quad\text{resp.}\quad C=\gk\left(3/2\,RV_2(\gk,\eta)\right)^{1/(\eta-1)}=\mathcal{O}(\gD_n^{1/(\eta-1)}).
\end{equation*}
This has to be considered in the simulation of $\widetilde\ell^{(n)}$ for different $\gD_n$ as we point out in the setting of Ass.~\ref{ass:1}:  
As shown in Theorem~\ref{thm:L^p}, the $L^p$-error $\E(|\ell(t)-\widetilde\ell^{(n)}(t)|^p)$ is bounded by
\begin{align*}
\E(|\ell(t)-\widetilde \ell^{(n)}(t)|^p)^{1/p}&\le2^{n} \Big(2Rp\zeta(\eta+1-p,C\eps^{-d}/2)+2R(C\eps^{-d}/2)^{p-1/d}+C^p\eps^{1-dp}\big)^{1/p}\\
&\quad+\frac{C^{1/p}_{\ell,T,p}2^{-n/p}}{2^{1/p}-1}.
\end{align*}
By substituting $\eps=C^{1/d}D^{-1/d}$, $d=1/{\eta}$ (see Remark~\ref{rem:Deps2}) and $2^n={T}/{\gD_n}$ one obtains 
\begin{equation} \label{L^p error}
\begin{split}
\E(|\ell(t)-\widetilde \ell^{(n)}(t)|^p)^{1/p}&\le \frac{(2Rp\zeta(\eta+1-p,D/2)+2R(D/2)^{p-\eta}+C^\eta D^{p-\eta})^{1/p}}{T^{-1}\gD_n}\\
&\quad+\frac{C^{1/p}_{\ell,T,p}2^{-n/p}}{2^{1/p}-1}
\end{split}
\end{equation}
With $R=\mathcal{O}(\gD_n)$ and $C=\mathcal{O}(\gD_n^{1/\eta})$ this implies with Ineq.~\eqref{L^p error} 
\begin{equation*}
\E(|\ell(t)-\widetilde \ell^{(n)}(t)|^p)^{1/p}=\mathcal{O}(D^{(p-\eta)/p}\gD_n^{1/p-1})+\mathcal{O}(\gD_n^{1/p}).
\end{equation*}
To equilibrate both error contributions, one may choose $D:=D_n=\gD_n^{p/(p-\eta)}$ in the simulation which leads to an $L^p$-error of order $\mathcal{O}(\gD_n^{1/p})$.
\end{rem}
As mentioned in the end of Section~\ref{sec:pre}, the one-dimensional processes $(\widetilde \ell_i, i=1,\ldots,N)$ in the spectral decomposition are not independent, but merely uncorrelated. 
In the next section we introduce a class of L\'evy fields for which uncorrelated processes can be obtained by subordinating a multi-dimensional Brownian motion.
Furthermore, for the simulation of these processes the Fourier inversion method may be employed and a bound for the constant $C_\ell$ (see Theorem~\ref{thm:H_error}) can be derived.  

\section{Generalized hyperbolic L\'evy processes}\label{sec:GH}
Distributions which belong to the class of generalized hyperbolic distributions may be used for a wide range of applications. GH distributions have been first introduced in~\cite{BN77} to model mass-sizes in aeolian sand (see also~\cite{BN78}). Since then they have been successfully applied, among others, in Finance and Biology. Giving a broad class the distributions are characterized by six parameters, famous representatives are the Student's t, the normal-inverse Gaussian, the hyperbolic and the variance-gamma distribution. The popularity of GH processes is explained by their flexibility in modeling various characteristics of a distribution such as asymmetries or heavy tails.
A further advantage in our setting is, that the characteristic function is known and, therefore, the Fourier Inversion may be applied to approximate these processes. This section is devoted to investigate several properties of multi-dimensional GH processes which are then used to construct an approximation of an infinite-dimensional GH field.
In contrast to the Gaussian case, the sum of two independent and possibly scaled GH processes is in general not again a GH process.
We show a possibility to approximate GH L\'evy fields via Karhunen-Lo\`eve expansions in such a way that the approximated field is itself again a GH L\'evy field. 
This is essential, so as to have convergence of the approximation to a GH L\'evy field in the sense of Theorem~\ref{thm:H_error}.
Furthermore, we give, for $N\in\N$, a representation of a $N$-dimensional GH process as a subordinated Brownian motion and show how a multi-dimensional GH process may be constructed from uncorrelated, one-dimensional GH processes with given parameters.
This may be exploited by the Fourier inversion algorithm in such a way that the computational expenses to simulate the approximated GH fields are virtually independent of the truncation index~$N$. 

Assume, for $N\in\N$, that $\lambda\in\R$, $\alpha>0$, $\beta\in\R^N$, $\delta>0$, $\mu\in\R^N$ and $\Gamma$ is a symmetric, positive definite (spd) $N\times N$-matrix with unit determinant.
We denote by $GH_N(\lambda,\alpha,\beta,\delta,\mu,\Gamma)$ the $N$-dimensional generalized hyperbolic distribution with probability density function 
\begin{equation}
f^{GH_N}(x;\gl,\ga,\gb,\gd,\mu,\gG)=\frac{\gg^\gl\ga^{N/2-\gl}}{(2\pi)^{N/2}\gd^\gl K_\gl(\gd\gg)}\frac{K_{\gl-N/2}(\ga g(x-\mu))}{g(x-\mu)^{N/2-\gl}}\exp(\beta'(x-\mu))
\end{equation}
for $x\in\R^N$, where 
\begin{equation*}
g(x):=\sqrt{\gd^2+x'\gG x}, \quad\gg^2:=\ga^2-\gb'\gG\gb
\end{equation*}
and $K_\gl(\cdot)$ is the modified Bessel-function of the second kind with $\gl$ degrees of freedom. The characteristic function of $GH_N(\lambda,\alpha,\beta,\delta,\mu,\Gamma)$  is given by
\begin{equation}\label{eq:gh_cf}
\begin{split}
\phi_{GH_N}(u; \gl,\ga,\gb,\gd,\mu,\gG)&:=\exp(iu'\mu)\Big(\frac{\ga^2-\gb'\gG\gb}{\ga^2-(iu+\gb)'\gG(iu+\gb)}\Big)^{\gl/2}\\
&\quad\cdot\frac{K_\gl(\gd(\ga^2-(iu+\gb)'\gG(iu+\gb))^{1/2})}{K_\gl(\gd\gg)},
\end{split}
\end{equation}
where $A'$ denotes the transpose of a matrix or vector $A$. 
For simplicity, we assume that the condition 
\begin{equation}\label{eq:ghcond}
\alpha^2>\beta'\Gamma\beta
\end{equation}
is satisfied\footnote{If $\alpha^2=\beta'\Gamma\beta$ and $\gl<0$, the distribution is still well-defined, but one has to consider the limit $\gg\to0^+$ in the Bessel functions, see~\cite{B81,R97}.}. If $N=1$, clearly, $\gG=1$ is the only possible choice for the "matrix parameter" $\gG$, thus we omit it in this case and denote the one-dimensional GH distribution by $GH(\lambda,\alpha,\beta,\delta,\mu)$. 
Barndorff--Nielsen obtains the GH distribution in \cite{BN78} as a normal variance-mean mixture of a $N$-dimensional normal distribution and a (one-dimensional) generalized inverse Gaussian (GIG) distribution with density function
\begin{equation*}
f^{GIG}(x;a,b,p)=\frac{(b/a)^p}{2K_p(ab)}x^{p-1}\exp(-\frac{1}{2}(a^2x^{-1}+b^2x)),\quad x>0
\end{equation*}
and parameters $a,b>0$ and $p\in\R$\footnote{The notation of the GIG distribution varies throughout the literature, we use the notation from~\cite{S03}.}.
To be more precise: Let $w^N(1)$ be a $N$-dimensional standard normally distributed random vector, $\gG$ a spd $N\times N$-structure matrix with unit determinant and $\ell^{GIG}(1)$ a $GIG(a,b,p)$ random variable, which is independent of $w^N(1)$. 
For $\mu,\beta\in\R^N$, we set $\gd=a$, $\gl=p$, $\ga=\sqrt{b^2+\gb'\gG\gb}$ and define the random variable $\ell^{GH_N}(1)$ as
\begin{equation}\label{eq:mean-variance}
\ell^{GH_N}(1):=\mu+\gG\gb \ell^{GIG}(1) + \sqrt{\ell^{GIG}(1)\gG}w^N(1).
\end{equation}
Then $\ell^{GH_N}(1)$ is $GH_N(\lambda,\alpha,\beta,\delta,\mu,\Gamma)$-distributed, where $\sqrt{\gG}$ denotes the Cholesky decomposition of the matrix $\gG$. 
With this in mind, one can draw samples of a GH distribution with given parameters by sampling multivariate normal and GIG-distributed random variables, 
as $a=\gd>0$ and $b=\sqrt{\ga^2-\gb'\gG\gb}>0$ is guaranteed by the conditions on the GIG parameters (this results in Eq.~\eqref{eq:ghcond} being fulfilled).

As noted in~\cite[Section 5]{E01}, for general $\gl\in\R$, we cannot assume that the increments of the GH L\'evy process (resp. of the subordinating process) over a time length other than one follow a GH distribution (resp. GIG distribution).
If $N=1$, however, the (one-dimensional) GH L\'evy process $\ell^{GH}$ has the representation 
\begin{equation}
\ell^{GH}(t)\stackrel{\cL}{=}\mu t+\beta \ell^{GIG}(t)+w(\ell^{GIG}(t)),\quad\text{for } t\geq 0,
\end{equation}
where $w$ is a one-dimensional Brownian motion and $\ell^{GIG}$ a GIG process independent of $w$ (see~\cite{CS02}). 
This result yields the following generalization:

\begin{lem}\label{lem:sub}
For $N\in\N$, the $N$-dimensional process $\ell^{GH_N}=(\ell^{GH_N}(t),t\in~ \T)$, which is $GH_N(\lambda,\alpha,\beta,\delta,\mu,\Gamma)$-distributed, can be represented as a subordinated $N$-dimensional Brownian motion $w^N$ via
\begin{equation}\label{sub1}
\ell^{GH_N}(t)\stackrel{\cL}{=}\mu t+\gG\beta \ell^{GIG}(t)+\sqrt{\gG}w^N(\ell^{GIG}(t)),
\end{equation}
where $(\ell^{GIG}(t),t\in\T)$ is a GIG L\'evy process independent of $w^N$ and $\sqrt{\gG}$ is the Cholesky decomposition of $\gG$. 
\end{lem}
\begin{proof} Since the $GH_N(\lambda,\alpha,\beta,\delta,\mu,\Gamma)$ distribution may be represented as a normal variance-mean mixture (see Eq.~\eqref{eq:mean-variance}), we have, that
$$\ell^{GH_N}(1)\stackrel{\cL}{=}\mu+\gG\gb \ell^{GIG}(1) + \sqrt{\gG \ell^{GIG}(1)}w^N(1)\stackrel{\cL}{=}\mu+\gG\gb \ell^{GIG}(1)+ \sqrt{\gG}w^N(\ell^{GIG}(1))$$
where $\ell^{GIG}(1)\sim GIG(\delta,\sqrt{\ga^2-\gb'\gG\gb},\gl)$ and $w^N$ is a $N$-dimensional Brownian motion independent of $\ell^{GIG}(1)$. 
The characteristic function of the mixed density is then given by  
\begin{equation*}
\phi_{GH_N}(u;\gl,\alpha,\beta,\delta,\mu,\Gamma)=\exp(iu'\mu)\cM_{GIG}(iu' \gG\gb-\frac{1}{2}u'\gG u;\delta,\sqrt{\ga^2-\gb'\gG\gb},\gl),
\end{equation*}
where $\cM_{GIG}$ denotes the moment generating function of $\ell^{GIG}(1)$ (see~\cite{BNJS82}). 
The GIG distribution is infinitely divisible, thus this GIG L\'evy process $\ell^{GIG}=(\ell^{GIG}(t),t\in\T)$ can be defined via its characteristic function for $t\in\T$:
$$\E(\exp(iu\ell^{GIG}(t)))=(\cM_{GIG}(iu;\delta,\sqrt{\ga^2-\gb'\gG\gb},\gl))^t.$$
The infinite divisibility yields further  
\begin{equation*}
\begin{split}
\E\big(\exp(iu'\ell^{GH_N}(t))\big)&=\E\big(\exp(iu'\ell^{GH_N}(1))\big)^t=(\phi_{GH}(u;\gl,\alpha,\beta,\delta,\mu,\Gamma))^t\\
&=\exp(iu'(\mu t))(\cM_{GIG}(iu' \gG\gb-\frac{1}{2}u'\gG u;\delta,\sqrt{\ga^2-\gb'\gG\gb},\gl))^t.
\end{split}
\end{equation*}
The expression above is the characteristic function of another normal variance-mean mixture, namely where the subordinator $\ell^{GIG}$ is a GIG process with characteristic function
$$\E(\exp(iu\ell^{GIG}(t)))=(\cM_{GIG}(iu;\delta,\sqrt{\ga^2-\gb'\gG\gb},\gl))^t.$$
Hence, $\ell^{GH_N}(t)$ can be expressed as 
\begin{equation*}
\ell^{GH_N}(t)\stackrel{\cL}{=}\mu t+\gG\beta\ell^{GIG}(t)+\sqrt\gG w^N(\ell^{GIG}(t)).
\end{equation*}
\end{proof}
\begin{rem}\label{rem:convolution}
In the special case of $\lambda=-\frac{1}{2}$ one obtains the normal inverse Gaussian (NIG) distribution. The mixing density is, in this case, the inverse Gaussian (IG) distribution. We denote the $N$-dimensional NIG distribution by $NIG_N(\alpha,\beta,\delta,\mu,\Gamma)$. This is the only subclass of GH distributions which is closed under convolutions in the sense that (see~\cite{OHH04})
$$NIG_N(\alpha,\beta,\delta_1,\mu_1,\Gamma)*NIG_N(\alpha,\beta,\delta_2,\mu_2,\Gamma)=NIG_N(\alpha,\beta,\delta_1+\delta_2,\mu_1+\mu_2,\Gamma).$$
For $\gl\in\R$, the sum of independent GH random variables is in general not GH-distributed.
This implies further, that one is in general not able to derive bridge laws of these processes in closed form, meaning we need to use the algorithms introduced in Section~\ref{ssec:Fourierinversion} for simulation.
\end{rem}

As shown in~\cite{BNH77}, the GH and the GIG distribution are infinitely-divisible, thus we can define the $N$-dimensional GH L\'evy process $\ell^{GH_N}=(\ell^{GH_N}(t),t\in\T)$ with characteristic function
$$\E(\exp(iu\ell^{GH_N}(t))=(\phi_{GH_N}(u;\lambda,\alpha,\beta,\delta,\mu,\Gamma))^t.$$
\begin{rem} \label{rem:NIG}
If $\gl=-\frac{1}{2}$, the corresponding NIG L\'evy process $(\ell^{NIG_N}(t),t\in\T)$ has characteristic function 
$$\E[\exp(iu\ell^{NIG_N}(t))]=(\phi_{GH_N}(u;-\frac{1}{2},\alpha,\beta,\delta,\mu,\Gamma))^t=\phi_{GH_N}(u;-\frac{1}{2},\alpha,\beta,t\delta,t\mu,\Gamma).$$
This is due to the fact that the characteristic function $\phi_{IG}(u;a,b)$ of the mixing IG distribution fulfills the identity 
\begin{equation*}
(\phi_{IG}(u;a,b))^t=\phi_{IG}(u;ta,b)
\end{equation*}
for any $t\in\T$ and $a,b>0$ (see \cite{S03}).
\end{rem}

We consider the finite time horizon $\T = [0,T]$, for $T<+\infty$, the probability space $(\gO,(\cA_t,t\geq 0),\mathbb{P})$, and a compact domain $\mathcal D\subset\R^s$ for $s\in\N$ to define a GH L\'evy field as a mapping
\begin{equation*}
L^{GH}:\gO\times \D\times\T\to\R,\quad(\go,x,t)\mapsto L^{GH}(\go)(x)(t),  
\end{equation*} 
such that for each $x\in \D$ the point-wise marginal process
\begin{equation*}
L^{GH}(\cdot)(x)(\cdot):\gO\times\T\to\R,\quad (\go,t)\mapsto L^{GH}(\go)(x)(t),  
\end{equation*} 
is a one-dimensional GH L\'evy process on $(\gO,(\cA_t,t\geq 0),\mathbb{P})$ with characteristic function
\begin{equation*}
\E\big(\exp(iuL^{GH}(x)(t))\big)=(\phi_{GH}(u;\gl(x),\alpha(x),\beta(x),\delta(x),\mu(x)))^t,
\end{equation*}
where the indicated parameters are given by continuous functions, i.e. $\gl,\gb,\mu\in C(\D;\R)$ and $\ga,\gd\in C(\D;\R_{>0})$. We assume that condition \eqref{eq:ghcond}, i.e. $\ga(x)^2>\gb(x)^2$, is fulfilled for any $x\in\D$  to ensure that $L^{GH}(x)(\cdot)$ is a well-defined GH L\'evy process.
This, in turn, means that $L^{GH}$ takes values in the Hilbert space $H=L^2(\mathcal D)$ and is square integrable as 
\begin{equation*} 
 \E(||L^{GH}(t)||_H^2)\le T\E(||L^{GH}(1)||_H^2)\le T \max\limits_{x\in\mathcal D}\E(L^{GH}(x)(1)^2)V_{\mathcal D},
\end{equation*}
where $V_{\mathcal D}$ denotes the volume of $\mathcal D$. 
The right hand side is finite since every GH distribution has finite variance (see for example~\cite{M04, S03}), the parameters of the distribution of $L^{GH}(x)(1)$ depend continuously on $x$ and $\mathcal D\subset\R^s$ is compact by assumption.
We use the Karhunen-Lo\`eve expansion from Section~\ref{sec:pre} to obtain an approximation of a given GH L\'evy field. 
For this purpose, we consider the truncated sum 
\begin{equation*}
 L^{GH}_N(x)(t) :=\sum_{i=1}^{N} \varphi_i(x)\ell_i^{GH}(t)\stackrel{\mathcal L}{=} \sum_{i=1}^{N} \varphi_i(x) \Big(\mu_i t+\beta_i \ell^{GIG}_i(t)+w_i(\ell^{GIG}_i(t))\Big),
\end{equation*}
where $N\in\N$ and $\varphi_i(x)=\sqrt{\rho_i}e_i(x)$ is the $i$-th component of the spectral basis evaluated at the spatial point $x$. 
For each $i,=1,\dots,N$, the processes $\ell_i^{GH}:=(\ell_i^{GH}(t), t\in\T)$ are uncorrelated but dependent $GH(\gl_i,\ga_i,\gb_i,\gd_i,\mu_i)$ L\'evy process.
From Theorem~\ref{thm:H_error} follows that $L^{GH}_N$ converges in $L^2(\gO;H)$ to $L^{GH}$ as $N\to\infty$. 
With given $\mu_i, \beta_i\in\R$, we have that 
	\begin{equation}\label{eq:Z_i}
	\ell^{GH}_i(t)\stackrel{\mathcal L}{=}\mu_i t+\beta_i \ell^{GIG}_i(t)+w_i(\ell^{GIG}_i(t)),
	\end{equation}
where for each $i$, the process $(\ell^{GIG}_i(t), t\in\T)$ is a GIG L\'evy process with parameters $a_i=\gd_i, b_i=(\ga_i^2-\gb_i^2)^{1/2}>0$ and $p_i=\gl_i\in\R$.
In addition, $(w_i(t), t\in\T)$ is a one-dimensional Brownian motion independent of $\ell^{GIG}_i$ and  all Brownian motions $w_1,\dots,w_N$ are mutually independent of each other, but the processes $\ell^{GIG}_1,\dots,\ell^{GIG}_N$ may be correlated.
We aim for an approximation $(L^{GH}_N(x)(t),t\in\T)$ which is a GH process for arbitrary $\varphi_i$ and $x\in\mathcal D$.
Remark~\ref{rem:convolution} suggests that this cannot be achieved by the summation of independent $\ell_i^{GH}$, but rather by using correlated subordinators $\ell^{GIG}_1,\dots,\ell^{GIG}_N$. 
Before we determine the correlation structure of the subordinators, we establish a necessary and sufficient condition on the $\ell_i^{GH}$ to achieve the desired distribution of the approximation. 

\begin{lem} \label{lem:GH_lin}
Let $N\in\N$, $t\in\T$ and $(\ell_i^{GH},i=1\dots,N)$ be GH processes as defined in Eq.~\eqref{eq:Z_i}. For a vector ${\bf a} = (a_1,\dots,a_N)$ with arbitrary numbers $a_1,\dots,a_N\in\R\setminus\{0\}$, the process $\ell^{GH,{\bf a}}$ defined by
\begin{equation}\label{eq:LNGH}
\ell^{GH,\bf a}(t):=\sum_{i=1}^{N} a_i\ell_i^{GH}(t)=\sum_{i=1}^{N} a_i(\mu_i+\beta_i\ell_i^{GIG}(t)+w_i(\ell_i^{GIG}(t)))
\end{equation}
is a one-dimensional GH process, if and only if the vector $\ell^{GH_N}(1):=(\ell_1^{GH}(1),\dots,\ell_N^{GH}(1))'$ is multivariate $GH_N(\gl^{(N)},\alpha^{(N)},\beta^{(N)},\delta^{(N)},\mu^{(N)},\Gamma)$-distributed with parameters $\gl^{(N)},\ga^{(N)}$, $\gd^{(N)}\in\R$, $\gb^{(N)},\mu^{(N)}\in\R^N$
 and structure matrix $\gG\in\R^{N\times N}$.
\end{lem}
The entries of the coefficient vector ${\bf a}$ in $\ell^{GH,\bf a}$ are later identified with the basis functions $\varphi_i(x)$ for $x\in\D$ to show that $L^{GH}_N(x)(\cdot)$ is a one-dimensional L\'evy process and the approximation $L^{GH}_N$ a $H$-valued GH L\'evy field.
\begin{proof}[Proof of Lemma~\ref{lem:GH_lin}] We first consider the case that 
	\begin{equation*}
	\ell^{GH_N}(1)\sim GH_N(\lambda^{(N)},\alpha^{(N)},\beta^{(N)},\delta^{(N)},\mu^{(N)},\Gamma).
	\end{equation*}
It is sufficient to show that $\ell^{GH,\bf a}(1)$ is a GH-distributed random variable, the infinite divisibility of the GH distribution then implies that $(\ell^{GH,\bf a}(t),t\in\T)$ is a GH process. 
Since the entries of the coefficient vector $a_1,\dots,a_N$ are non-zero, there exists a non-singular $N\times N$ matrix $A$, such that $\ell^{GH,\bf a}(1)$ is the first component of the vector $A\ell^{GH_N}(1)$. 
If $\ell^{GH_N}(1)$ is multi-dimensional GH-distributed, then follows from~\cite[Theorem 1]{B81}, that $A\ell^{GH_N}(1)$ is also multi-dimensional GH-distributed and that the first component of $A\ell^{GH_N}(1)$, namely $\ell^{GH,\bf a}(1)$, follows a one-dimensional GH distribution (the parameters of the distribution of $\ell^{GH,\bf a}(1)$ depend on $A$ and on $\lambda^{(N)},\alpha^{(N)},\beta^{(N)},\delta^{(N)},\mu^{(N)},\Gamma$ and are explicitly given in \cite{B81} and below). \\
On the other hand, assume that $\ell^{GH,\bf a}(1)$ is a GH random variable (with arbitrary coefficients), but $\ell^{GH_N}(1)$ is not $N$-dimensional GH-distributed. 
This means there is no representation of $\ell^{GH_N}(1)$ such that
\begin{equation*}
\ell^{GH_N}(1)\stackrel{\cL}{=}\mu+\gG\gb\ \ell^{GIG}(1)+\sqrt{\gG}w^N(\ell^{GIG}(1))\
\end{equation*}  
with $\mu,\gb\in\R^N$, $\gG\in\R^{N\times N}$ spd with determinant one, a GIG random variable $\ell^{GIG}(1)$ and a $N$-dimensional Brownian motion $w^N$ independent of $\ell^{GIG}(1)$. 
This implies that $\ell^{GH,\bf a}(1)=(A\ell^{GH_N}(1))_1$ has no representation 
\begin{equation*}
\begin{split}
\ell^{GH,\bf a}(1)&=(A\mu)_1+(A\gG\gb)_1\ell^{GIG}(1)+(A\sqrt{\gG}w^N(\ell^{GIG}(1)))_1\\
&\stackrel{\cL}{=}(A\mu)_1+(A\gG\gb)_1\ell^{GIG}(1)+\sqrt{\ell^{GIG}(1)A_{[1]}\gG A_{[1]}'}w^1(1),
\end{split}
\end{equation*}  
where $A_{[1]}$ denotes the first row of the matrix $A$ and $w^1(1)\sim\cN(0,1)$. 
For the last equality we have used the affine linear transformation property of multi-dimensional normal distributions and that $\gG$ is positive definite. 
Since $c_A:=A_{[1]}\gG A_{[1]}'>0$, we can divide the equation above by $\sqrt{c_A}$ and obtain that $c_A^{-1/2}\,\ell^{GH,\bf a}(1)$ cannot be a GH-distributed random variable, as it cannot be expressed as a normal variance-mean mixture with a GIG-distribution. 
But this is a contradiction, since $\ell^{GH,\bf a}(1)$ is GH-distributed by assumption and the class of GH distributions is closed under regular affine linear transformations (see~\cite[Theorem 1c]{B81}).
\end{proof} 
\begin{rem}\label{rem:coefficients}
The condition $a_i\neq0$ is, in fact, not necessary in Lemma~\ref{lem:GH_lin}. If, for $k\in\{1,\dots,N-1\}$, $k$ coefficients $a_{i_1}=\dots=a_{i_k}=0$, then the summation reduces to 
\begin{equation*}
\ell^{GH,\bf a}(t)=\sum_{i=1}^{N}a_i \ell_i^{GH}(t)=\sum_{l=1}^{N-k}a_{j_l}\ell^{GH}_{j_l}(t),
\end{equation*}
where the indices $j_l$ are chosen such that $a_{j_l}\neq0$ for $l=1,\dots,N-k$. If $P\in\R^{N\times N}$ is the permutation matrix with 
\begin{equation*}
P\ell^{GH_N}(1)=P(\ell_1^{GH}(1),\dots,\ell_N^{GH}(1))'=( \ell^{GH}_{j_1}(1),\dots, \ell^{GH}_{j_{N-k}}(1), \ell^{GH}_{i_1}(1),\dots, \ell^{GH}_{i_k}(1))', 
\end{equation*}
then $P\ell^{GH_N}$ is again $N$-dimensionally GH-distributed and by~\cite[Theorem 1a]{B81} the vector $(\ell^{GH}_{j_1}(1),\dots,\ell^{GH}_{j_{N-k}}(1))$ admits a $(N-k)$-dimensional GH law. Thus, we only consider the case where all coefficients are non-vanishing.
\end{rem}

The previous proposition states that the KL approximation
\begin{equation*}
 L^{GH}_N(x)(t) = \sum_{i=1}^{N} \varphi_i(x) \ell_i^{GH}(t),
\end{equation*}
can only be a GH process for arbitrary $(\varphi_i(x),i=1,\dots,N)$ if the $\ell_i^{GH}$ are correlated in such a way that they form a multi-dimensional GH process. This rules out the possibility of independent processes $(\ell_i^{GH},i=1,\ldots,N)$, because if $\ell^{GH_N}(1)$ is multi-dimensional GH-distributed, it is not possible that the marginals $\ell_i^{GH}(1)$ are independent GH-distributed random variables (see \cite{B81}).
The parameters $\gl_i,\ga_i,\gb_i,\gd_i,\mu_i$ of each process $\ell_i^{GH}$ should remain as unrestricted as possible, so we determine in the next step the parameters of the marginals of a $GH_N(\lambda^{(N)},\alpha^{(N)},\beta^{(N)},\delta^{(N)},\mu^{(N)},\Gamma)$ distribution and show how the subordinators $(\ell^{GIG}_i,i=1,\ldots,N)$ might be correlated. For this purpose, we introduce the notation $A\invtr :=(A^{-1})'$ if $A$ is an invertible square matrix. The following result allows us to determine the marginal distributions of a $N$-dimensional GH distribution. 

\begin{lem}(Masuda~\cite{M04}, who refers to \cite{BJ81}, Lemma A.1.)\label{lem:GH_N}
Let 
$$\ell^{GH_N}(1)=(\ell_1^{GH}(1),\dots,\ell_N^{GH}(1))'\sim GH_N(\lambda^{(N)},\alpha^{(N)},\beta^{(N)},\delta^{(N)},\mu^{(N)},\Gamma),$$ 
then for each $i$ we have that $\ell_i^{GH}(1)\sim GH(\lambda_i,\alpha_i,\beta_i,\delta_i,\mu_i)$, where 
\begin{equation*}
\begin{split}
&\lambda_i=\gl^{(N)},\quad\alpha_i=\gG_{ii}^{-1/2}\left[(\ga^{(N)})^2-\gb_{-i}'\left(\gG_{-i,22}
-\gG_{-i,21}\gG_{ii}^{-1}\gG_{-i,12}\right)\gb_{-i}\right]^{1/2}\\
&\beta_i=\gb^{(N)}_i+\gG_{ii}^{-1}\gG_{-i,12}\gb_{-i},\quad\delta_i=\sqrt{\gG_{ii}}\gd^{(N)}_i,\quad\mu_i=\mu_i^{(N)},
\end{split}
\end{equation*} 
together with
\begin{equation*}
\begin{split}
&\gb_{-i}:=(\gb_1^{(N)},\dots,\gb_{i-1}^{(N)},\gb_{i+1}^{(N)},\dots,\gb_N^{(N)})',\\
&\gG_{-i,12}:=(\gG_{i,1},\dots,\gG_{i,i-1},\gG_{i,i+1},\dots,\gG_{i,N}),\quad \gG_{-i,21}:=\gG_{-i,12}'
\end{split}
\end{equation*} 
and $\gG_{-i,22}$ denotes the $(N-1)\times(N-1)$ matrix which is obtained by removing the $i$-th row and column of $\gG$.
\end{lem}

Assume that $\ell^{GH_N}(1)\sim GH_N(\lambda^{(N)},\alpha^{(N)},\beta^{(N)},\delta^{(N)},\mu^{(N)},\Gamma)$, since this is a necessary (and sufficient) condition so that the (truncated) KL expansion is a GH process. 
Lemma~\ref{lem:GH_N} gives immediately, that for all $i=1,\dots,N$, the parameters $\gl_i=\gl^{(N)}$ have to be identical, whereas the drift $\mu_i$ may be chosen arbitrary for each process $\ell_i^{GH}$. Furthermore, the expectation and  covariance matrix of $\ell^{GH_N}(1)$ is given by 
\begin{equation} \label{eq:EZ}
\E(\ell^{GH_N}(1))=\mu^{(N)}+\frac{\gd^{(N)}K_{\gl^{(N)}+1}(\gd^{(N)}\gg^{(N)})}{\gg^{(N)}K_{\gl^{(N)}}(\gd^{(N)}\gg^{(N)})}\gG\gb^{(N)}
\end{equation}
and
\begin{equation} \label{eq:VarZ}
\begin{split}
\text{Var}(\ell^{GH_N}(1))&=\frac{\gd^{(N)}K_{\gl^{(N)}+1}(\gd^{(N)}\gg^{(N)})}{\gg^{(N)}K_{\gl^{(N)}}(\gd^{(N)}\gg^{(N)})}\gG
+\Big(\frac{\gd^{(N)}}{\gg^{(N)}}\Big)^2(\gG\gb^{(N)})(\gG\gb^{(N)})'\\
&\qquad\qquad\cdot\Bigg(\frac{K_{\gl^{(N)}+2}(\gd^{(N)}\gg^{(N)})}{K_{\gl^{(N)}}(\gd^{(N)}\gg^{(N)})}-\frac{K^2_{\gl^{(N)}+1}(\gd^{(N)}\gg^{(N)})}{K^2_{\gl^{(N)}}(\gd^{(N)}\gg^{(N)})}\Bigg),
\end{split}
\end{equation}
where $\gg^{(N)}:=((\ga^{(N)})^2-\gb^{(N)}~\!\!'\gG\gb^{(N)})^{1/2}$ (see \cite{M04}). 

\begin{ex} \label{ex1}
Consider the case that the processes $\ell_1^{GH},\dots,\ell_N^{GH}$ are generated by the same subordinating $GIG(a,b,p)$ process $\ell^{GIG}$, i.e. 
\begin{equation*}
\ell_i^{GH}(t)=\mu_it+\gb_i\ell^{GIG}(t)+w_i(\ell^{GIG}(t)).
\end{equation*}
Then $\ell_i^{GH}(1)\sim GH(\gl,\ga_i,\gb_i,\gd,\mu_i)$, where $\gl=p$, $\gd=a$ are independent of $i$ and $\ga_i=(b^2+\gb_i^2)^{1/2}$. If $\mu^{(N)}:=(\mu_1\dots,\mu_N)'$, $\gb^{(N)}:=(\gb_1,\dots,\gb_N)'$ and $\gG$ is the $N\times N$ identity matrix, then 
\begin{equation*}
\begin{split}
\ell^{GH_N}(t)=(\ell_1^{GH}(t),\dots,\ell_N^{GH}(t))'&\stackrel{\cL}{=}\mu t+\gb \ell^{GIG}(t)+w^N(\ell^{GIG}(t))\\
&=\mu t+\gG\gb \ell^{GIG}(t)+\sqrt{\gG}w^N(\ell^{GIG}(t)),
\end{split}
\end{equation*}
where $w^N$ is a $N$-dimensional Brownian motion independent of $\ell^{GIG}$. Hence, $\ell^{GH_N}(t)$ is a multi-dimensional $GH_N(\gl,\ga^{(N)},\gb^{(N)},\gd,\mu^{(N)}$,$\gG)$ process with $\ga^{(N)}=\sqrt{b^2+\gb'\gb}$. 
One checks using Lemma~\ref{lem:GH_N} that the parameters of the marginals of $\ell^{GH_N}(1)$ and $\ell_i^{GH}(1)$ coincide for each $i$, and that expectation and covariance of $\ell^{GH_N}(1)$ are given by Eq.~\eqref{eq:EZ} and Eq.~\eqref{eq:VarZ}.

By Lemma~\ref{lem:GH_lin}, we have that the Karhunen-Lo\`eve expansion
\begin{equation*}
L_N^{GH}(x)(t)=\sum_{i=1}^N\varphi_i(x)\ell_i^{GH}(t)
\end{equation*}
in this example is a GH process for each $x\in\mathcal{D}$ and arbitrary basis functions $(\varphi_i,i=1,\ldots,N)$.
\end{ex}

\begin{rem}\label{rem:subordinators}
 Lemma~\ref{lem:GH_N} dictates that the subordinators $(\ell^{GIG}_i,i=1,\ldots,N)$ cannot be independent. In Example~\ref{ex1} fully correlated subordinators were used. A different way to correlate the subordinators, so that Lemma~\ref{lem:GH_N} is fulfilled, would lead to a correlation matrix, just being multiplied with $\Gamma$. For simplicity, in the remainder of the paper, especially for the numerical examples in Section~\ref{sec:num}, we use fully correlated subordinators. 
\end{rem}

As shown in \cite[Theorem 1c]{B81} the class of $N$-dimensional GH distributions is closed under regular linear transformations: If $N\in\N$, $\ell^{GH_N}(1)\sim GH_N(\lambda,\alpha,\beta,\delta,\mu,\Gamma)$, $A$ is an invertible $N\times N$-matrix and $b\in\R^N$, then the random vector $A\ell^{GH_N}(1)+b$ has distribution
$$GH_N(\lambda,||A||^{-1/N}\alpha,A\invtr\beta,||A||^{1/N}\delta,A\mu+b,||A||^{-2/N}A\Gamma A'),$$
where $||A||$ denotes the absolute value of the determinant of $A$. With this and the assumption $\ell^{GH_N}(1)\sim GH_N(\lambda^{(N)},\alpha^{(N)},\beta^{(N)},\delta^{(N)},\mu^{(N)},\Gamma)$, we are also able to determine the point-wise law of $L_N^{GH}$ for given coefficients $\varphi_1(x),\dots,\varphi_N(x)$. 

\begin{lem} \label{lem:KL_GH}
Let $\ell^{GH_N}(1)\sim GH_N(\lambda^{(N)},\alpha^{(N)},\beta^{(N)},\delta^{(N)},\mu^{(N)},\Gamma)$ and for $x\in \D$ let  $(\varphi_i(x),i = 1,\dots,N)$ be a sequence of non-zero coefficients (see Remark~\ref{rem:coefficients}). 
Then $(L_N^{GH}(x)(t),t\in\T)$ is a GH L\'evy process with parameters depending on $x$.
\end{lem}
\begin{proof}
It is again sufficient to show that $L_N^{GH}(x)(1)$ follows a GH law, the resulting parameters are given below. For $x\in \D$, define the $N\times N$ matrix $A(x)$ via
\begin{equation} \label{eq:A}
A(x)_{ij}:=
\begin{cases}
     \varphi_j(x) & \text{if $i=1$ or if $i=j$}   \\
     0  & \text{elsewhere}  
   \end{cases}.
\end{equation}
The matrix $A(x)$ is invertible with determinant $\prod_{i=1}^N\varphi_i(x)\neq0$ and inverse $A(x)^{-1}$ given by
\begin{equation*}
A(x)^{-1}_{ij}:=
\begin{cases}
     -\varphi_1(x)^{-1} & \text{if $i=1$ and $j\ge2$}   \\
     \varphi_i(x)^{-1} & \text{if $i=j$}   \\
     0  & \text{elsewhere}  
   \end{cases}.
\end{equation*}
Then $L_N^{GH}(x)(1)=\sum_{i=1}^N\varphi_i(x)\ell_i^{GH}(1)$ is the first entry of the random vector $A(x)\ell^{GH_N}(1)$. 
By the affine transformation property of the GH distribution and Lemma~\ref{lem:GH_N} it follows that $L_N^{GH}(x)(1)$ is one-dimensional GH-distributed. 
Now define $\widetilde{\gG}:=A(x)\gG A(x)'$, the partition
\begin{equation*}
\widetilde{\gG}=\begin{pmatrix} \widetilde{\gG}_{11} & \widetilde{\gG}_{2,1}'\\ \widetilde{\gG}_{2,1} & \widetilde{\gG}_{2,2} \end{pmatrix}
\end{equation*}
such that $\widetilde{\gG}_{2,1}\in\R^{N-1}$ and $\widetilde{\gG}_{2,2}\in\R^{(N-1)\times(N-1)}$ and the vector
$$\widetilde{\gb}:=\left(\gb_2^{(N)}\varphi_2(x)^{-1}-\gb_1^{(N)}\varphi_1(x)^{-1},\dots,\gb_N^{(N)}\varphi_N(x)^{-1}-\gb_1^{(N)}\varphi_1(x)^{-1}\right)'\in\R^{N-1}.$$
The parameters $\lambda_L,\alpha_L(x),\beta_L(x),\delta_L(x)$ and $\mu_L(x)$ of $L_N^{GH}(x)$ are then given by
\begin{align*}
\gl_L&=\gl^{(N)},\\ 
\ga_L(x)&=\widetilde{\gG}_{11}^{-1/2}\left[(\ga^{(N)})^2-\widetilde{\gb}'(\widetilde{\gG}_{2,2}-
\widetilde{\gG}_{11}^{-1}\widetilde{\gG}_{2,1}\widetilde{\gG}_{2,1}')\widetilde{\gb} \right]^{1/2},\\
\gd_L(x)&=\gd^{(N)}\sqrt{\widetilde{\gG}_{11}}=\delta^{(N)}\Big(\sum_{i,j=1}^N\varphi_i(x)\varphi_j(x)\gG_{ij}\Big)^{1/2},\\
\gb_L(x)&=\gb_1^{(N)}\varphi_1(x)^{-1}+\widetilde{\gG}_{11}^{-1}\widetilde{\gG}_{2,1}'\widetilde{\gb}\quad\text{and}\\
\mu_L(x)&=[A(x)\mu^{(N)}]_1=\sum_{i=1}^N\varphi_i(x)\mu_i^{(N)}.
\end{align*}
\end{proof}
To ensure $L^2(\gO;\R)$ convergence as in Theorem~\ref{thm:H_error} of the series
\begin{equation*}
\widetilde L_N^{GH}(x)(t) = \sum_{i=1}^N\sqrt{\rho_i}e_i(x)\widetilde\ell_i^{GH}(t),
\end{equation*}
we need to simulate approximations of uncorrelated, one-dimensional GH processes $\ell_i^{GH}$ with given parameters 
$\ell_i^{GH}(1)\sim GH(\gl_i, \ga_i,\gb_i, \gd_i,\mu_i)$.
To obtain a sufficiently good approximation of the L\'evy field, $N$ is coupled to the time discretization of $\T$ and the decay of the eigenvalues of $Q$ (see Remark~\ref{rem:trunc}). 
The simulation of a large number $N$ of independent GH processes is computationally expensive, so we focus on a different approach.
Instead of generating $N$ dependent but uncorrelated, one-dimensional processes, we generate one $N$-dimensional process with decorrelated marginals.
For this approach to work we need to impose some restrictions on the target parameters $\gl_i, \ga_i,\gb_i$ and $\gd_i$.

\begin{thm}\label{thm:Z_U}
Let $(\ell_i^{GH},i=1\ldots,N)$ be one-dimensional GH processes, where, for $i=1,\dots,N$,
$\ell_i^{GH}(1)\sim GH(\gl_i, \ga_i,\gb_i, \gd_i,\mu_i).$
The vector $\ell^{GH_N}:=(\ell_1^{GH},\dots,\ell_N^{GH})'$ is only a $N$-dimensional GH process if there are constants $\lambda\in\R$ and $c>0$ such that for any $i$
\begin{equation*}
\gl_i=\gl\quad\text{and}\quad \gd_i(\ga_i^2-\gb_i^2)^{1/2}=c.
\end{equation*}
If, in addition, the symmetric matrix $U\in\R^{N\times N}$ defined by
\begin{equation*}
U_{ij}:=
\begin{cases}
\gd_i^2 &\text{if $i=j$}\\
\frac{K_{\gl+1}(c)^2-K_{\gl+2}(c)K_{\gl}(c)}{K_{\gl+1}(c)K_{\gl}(c)}\frac{\gb_i\gd_i^2\gb_j\gd_j^2}{c} &\text{if $i\neq j$}
\end{cases},
\end{equation*}
is positive definite, it is possible to construct a $N$-dimensional GH process $\ell^{GH_N,U}$ with uncorrelated marginals $\ell^{GH,U}_i$ and 
\begin{equation*}
 \ell^{GH,U}_i(1)\stackrel{\mathcal L}{=} \ell_i^{GH}(1)\sim GH(\gl_i, \ga_i,\gb_i, \gd_i,\mu_i).
\end{equation*}
\end{thm}
\begin{proof}
We start with the necessary condition to obtain a multi-dimensional GH distribution. Let $\ell^{GH_N}$ be a $N$-dimensional GH process with
$$\ell^{GH_N}(1)\sim GH_N(\lambda^{(N)},\alpha^{(N)},\beta^{(N)},\delta^{(N)},\mu^{(N)},\Gamma).$$ 
If the law of the marginals of $\ell^{GH_N}$ is denoted by $ \ell_i^{GH}(1)\sim GH(\gl_i, \ga_i,\gb_i, \gd_i,\mu_i),$
then one sees immediately from Lemma~\ref{lem:GH_N} that $\gl_i=\gl^{(N)}$ and $\mu_i=\mu_i^{(N)}$ for all $i=1,\dots,N$.
 With the equations for $\gb_i$ and $\gd_i$ from Lemma~\ref{lem:GH_N}, we derive for $\gG\gb^{(N)}$
\begin{equation}\label{eq:gGb}
(\gG\gb^{(N)})_i=\gG_{ii}\gb_i^{(N)}+\sum_{k=1, k\neq i}^N \gG_{ik}\gb_k^{(N)}=\gG_{ii}\gb_i^{(N)}+\gG_{ii}(\gb_i-\gb_i^{(N)})=\big(\frac{\gd_i}{\gd^{(N)}}\big)^2\gb_i,
\end{equation}
which leads to
\begin{align*}
\ga_i^2&=\gG_{ii}^{-1}(\ga^{(N)})^2-\gG_{ii}^{-1}\sum_{k=1,k\neq i}^N\gb_k^{(N)}\sum_{l=1,l\neq i}^N\gG_{kl}\gb_l^{(N)}+\big(\gG_{ii}^{-1}\sum_{k=1, k\neq i}^N \gG_{ik}\gb_k^{(N)}\big)^2\\
&=\big(\frac{\gd^{(N)}\ga^{(N)}}{\gd_i}\big)^2-\big(\frac{\gd^{(N)}}{\gd_i}\big)^2\sum_{k=1,k\neq i}^N\gb_k^{(N)}((\gG\gb^{(N)})_k-\gG_{ik}\gb_i^{(N)})+(\gb_i-\gb_i^{(N)})^2\\
&=\big(\frac{\gd^{(N)}\ga^{(N)}}{\gd_i}\big)^2-\sum_{k=1,k\neq i}^N\gb_k^{(N)}\frac{\gd_k^2}{\gd_i^2}\gb_k\\
&\quad+\big(\frac{\gd^{(N)}}{\gd_i}\big)^2\gb_i^{(N)}((\gG\gb^{(N)})_i-\gG_{ii}\gb_i^{(N)})+\gb_i^2-2\gb_i\gb_i^{(N)}+(\gb_i^{(N)})^2\\
&=\big(\frac{\gd^{(N)}\ga^{(N)}}{\gd_i}\big)^2-\sum_{k=1}^N\gb_k^{(N)}\frac{\gd_k^2}{\gd_i^2}\gb_k+\gb_i^2.
\end{align*}
The last equation is equivalent to 
\begin{equation} \label{eq:gamma}
\gd_i^2(\ga_i^2-\gb_i^2)=(\gd^{(N)}\ga^{(N)})^2-\sum_{k=1}^N\gb_k^{(N)}\underbrace{\gd_k\gb_k}_{=(\gd^{(N)})^2(\Gamma\gb_k^{(N)})}=(\gd^{(N)})^2\left((\ga^{(N)})^2-\gb^{(N)}~\!\!'\gG\gb^{(N)}\right),
\end{equation}
and since the right hand side does not depend on $i$, we get that $\gd_i^2(\ga_i^2-\gb_i^2)>0$ has to be independent of $i$.

Now assume we have a set of parameters $((\gl_i,\ga_i,\gb_i,\gd_i,\mu_i),i=1,\dots,N)$ with 
\begin{equation*}
\gd_i\sqrt{\ga_i^2-\gb_i^2}=c>0\quad\text{and}\quad\gl_i=\gl\in\R,
\end{equation*}
where $c$ and $\gl$ are independent of the index $i$. 
Furthermore, let the matrix $U$ as defined in the claim be positive definite.
We show how parameters $\gl^{(U)}, \ga^{(U)}, \gb^{(U)}, \gd^{(U)},  \mu^{(U)}$ and $\gG^{(U)}$ of a $N$-dimensional GH process $\ell^{GH_N,U}$ may be chosen, such that its marginals are uncorrelated with law $\ell^{GH,U}_i(1)\sim GH(\gl,\ga_i,\gb_i,\gd_i,\mu_i)$. Clearly, we have to set $\gl^{(U)}:=\gl$ and $\mu^{(U)}:=(\mu_1,\dots,\mu_N)'$. Eq.~\eqref{eq:gGb} and Eq.~\eqref{eq:gamma} yield the conditions
$$(\gd^{(U)})^2(\gG^{(U)}\gb^{(U)})_i=\gd_i^2\gb_i \quad\text{and}\quad \gd^{(U)}\sqrt{(\ga^{(U)})^2-\gb^{(U)T}\gG^{(U)}\gb^{(U)}}=\gd_i\sqrt{\ga_i^2-\gb_i^2}=c.$$
If $(\gd^{(U)})^2\gG^{(U)}$ fulfills the identity $(\gd^{(U)})^2\gG^{(U)}=U$, we get by Eq.~\eqref{eq:VarZ} for $i\neq j$
\begin{align*}
\text{Cov}(\ell^{GH,U}_i(1),\ell^{GH,U}_j(1))&=\frac{K_{\gl+1}(c)}{cK_{\gl}(c)}(\gd^{(U)})^2\gG_{ij}^{(U)}\\
&\quad+\frac{K_{\gl+2}(c)K_{\gl}(c)-K_{\gl+1}^2(c)}{c^2K_{\gl}(c)^2}((\gd^{(U)})^2\gG^{(U)}\gb^{(U)})_i((\gd^{(U)})^2\gG^{(U)}\gb^{(U)})_j\\
&=\frac{K_{\gl+1}(c)}{cK_{\gl}(c)}U_{ij}+\frac{K_{\gl+2}(c)K_{\gl}(c)-K_{\gl+1}^2(c)}{c^2K_{\gl}(c)^2}\gd_i^2\gb_i\gd_j^2\gb_j=0,
\end{align*}
hence all marginals are uncorrelated. To obtain a well-defined $N$-dimensional GH distribution, we still have to make sure that $\gG^{(U)}$ is spd with unit determinant. 
If we define $\gd^{(U)}:=(\text{det}(U))^{1/(2N)}$, then $\gd^{(U)}>0$ (since $\text{det}(U)>0$ by assumption) and $\gG^{(U)}=(\gd^{(U)})^{-2}U$ is spd with $\text{det}(\gG^{(U)})=1$. It remains to determine appropriate parameters $\ga^{(U)}>0$ and $\gb^{(U)}\in\R^N$. 
For $\gb^{(U)}$, we use once again Lemma~\ref{lem:GH_N} to obtain the linear equations
\begin{equation*}
\gb_i=\gb_i+(\gG_{ii}^{(U)})^{-1}\sum_{k=1,k\neq i}\gG_{ik}^{(U)}\gb_k^{(U)},
\end{equation*}
for $i=1,\dots,N$. The corresponding system of linear equations is given by
\begin{equation*}
\begin{pmatrix}
(\gG_{11}^{(U)})^{-1}\\
&\ddots\\
&&(\gG_{NN}^{(U)})^{-1}\\
\end{pmatrix}
\gG^{(U)}\gb^{(U)}=
\left(\begin{array}{c} \gb_1\\ \vdots \\ \gb_N \end{array}\right), 
\end{equation*}
and has a unique solution $\gb^{(U)}$ for any right hand side $(\gb_1,\dots,\gb_N)'$, because $\gG^{(U)}$ as constructed above is invertible with positive diagonal entries.
Finally, we are able to calculate $\ga^{(U)}$ via Equation~\eqref{eq:gamma} as
\begin{equation*}
\ga^{(U)}=\Big(\sum_{k=1}^N\gd_k^2\gb_k\gb_k^{(U)}+\big(\frac{c}{\gd^{(U)}}\big)^2\Big)^{1/2}=\Big(\gb^{(U)}~\!\!'\Gamma^{(U)}\gb^{(U)}+\big(\frac{c}{\gd^{(U)}}\big)^2\Big)^{1/2}
\end{equation*}
 and obtain the desired marginal distributions.
\end{proof}
Note that the KL-expansion $L^{GH}_N(x)(\cdot)$ generated by $(\ell^{GH,U}_i,i=1\ldots,N)$ in Theorem~\ref{thm:Z_U} is a GH process for each $x\in \D$ by Lemma~\ref {lem:KL_GH}, whereas this is not the case if the processes $(\ell_i^{GH},i=1,\ldots,N)$ are generated independently of each other:
By Lemma~\ref{lem:GH_lin} we have that $L^{GH}_N(x)(1)$ is only GH distributed if the vector $(\ell_1^{GH}(1),\dots,\ell_N^{GH}(1))'$ admits a multi-dimensional GH law. 
As noted in~\cite{B81} after Theorem 1, this is impossible if the processes (and hence $(\ell_i^{GH}(1),i=1,\ldots,N)$) are independent.
Whenever Theorem~\ref{thm:Z_U} is applicable, we are able to approximate a GH L\'evy field by generating a $N$-dimensional GH processes, where $N$ is the truncation index of the KL expansion.
To this end, Lemma~\ref{lem:sub} suggests the simulation of GIG processes and then subordinating $N$-dimensional Brownian motions.
With this simulation approach the question arises on why we have taken a detour via the subordinating GIG process instead of using the characteristic function a of GH process in Equation~\eqref{eq:gh_cf} for a ``direct'' simulation. 
This has several reasons: First, the approximation of the inversion formula \eqref{eq:finv} can only be applied for one-dimensional GH processes, where the costs of evaluating $\phi_{GH}$ or $\phi_{GIG}$ are roughly the same. 
In comparison, the costs of sampling a Brownian motion are negligible.
Second, in the multi-dimensional case, we need that all marginals of the GH process are generated by the same or correlated subordinator(s), which leaves us no choice but to sample the underlying GIG process. 
In addition, the simulation of a GH field requires in some cases only one subordinating process to generate a multi-dimensional GH process with uncorrelated marginals (see Theorem~\ref{thm:Z_U}). 
This approach is in general more efficient than sampling a large number of uncorrelated, one-dimensional GH processes for the KL expansion.
As we demonstrate in the following section, it is a straightforward application of the Fourier inversion algorithm to approximate a GIG process $\ell^{GIG}$ with given parameters, since all necessary assumptions are fulfilled and the bounding parameters $\eta, R, \theta$ and $B$ may readily be calculated.

\section{Numerics}\label{sec:num}
In this section we provide some details on the implementation of the Fourier inversion method.
Thereafter, we apply this methodology to approximate a GH L\'evy field and conclude with some numerical examples.
\subsection{Notes on implementation}\label{sec:imp}
Suppose we simulate a given one-dimensional L\'evy process $\ell$ which fulfills Assumption~\ref{ass:1} resp. Assumption~\ref{ass:2}, using the step size $\gD_n>0$ and characteristic function $(\phi_\ell)^{\gD_n}$. 
Usually the parameter $\eta$ cannot be chosen arbitrary high (as for the GIG process), but it may be possible to choose $\eta$ within a certain range, for instance $\eta\in(1,2]$ for the Cauchy process in Example~\ref{ex:eta_cond}.
As a rule of thumb, $\eta$ should always be determined as large as possible, as the convergence rates in Theorems~\ref{thm:L^p 1} and~\ref{thm:L^p 2} directly depend on $\eta$.
In addition, we concluded in Remark~\ref{rem:t-conv} that $D\simeq\gD_n^{p/(p-\eta)}$ is an appropriate choice to guarantee an $L^p$-error of order $\mathcal{O}(\gD_n^{1/p})$.
This means that for a given $p$, $D$ decreases as $\eta$ increases. 
Since the number of summations $M$ in Algorithm~\ref{ass:2} depends on $D$ (see Theorem~\ref{thm:CDF_approx}), an increasing parameter $\eta$ also reduces computational time. 
Once $\eta$ is determined, we derive $R$ by differentiation of $(\phi_\ell)^{\gD_n}$ as in Remark~\ref{rem:2pi}.
Similarly to $\eta$, it is often possible to choose between several values of $\theta>0$, but it is difficult to give a-priori a recommendation on how $\theta$ should be selected. 
One rather calculates for several admissible $\theta$ the constant $C_\theta:=\max_{u\in\R}|u^\theta(\phi_\ell(u))^{\gD_n}|$ numerically and deducts $B_\theta=(2\pi)^{-\theta}C_\theta$.
Each combination of $(\theta,B_\theta)$ then results in a valid number of summations $M_\theta$ in the discrete Fourier Inversion algorithm. 
Since $\theta$ and $B_\theta$ are only necessary to determine $M_\theta$, we may simply use the smallest $M_\theta$ for the simulation.
To find $\widetilde X$ with $\widetilde F(\widetilde X)=U$ in Algorithm~\ref{algo:approx2}, we use a globalized Newton method with \textit{backtracking~line~search}, also known as \textit{Armijo increment control}.
The step lengths during the line search are determined by interpolation, which is a robust technique if combined with a standard Newton method.
Details on the globalized Newton method with backtracking may be found, for example, in \cite{NW06}, an example how the algorithm is used is given in \cite{NR}.
Although convergence of this root finding algorithm is ensured by the increment control, its efficiency depends heavily on the choice of the initial value $\widetilde X_0$. 
Clearly,  $\widetilde X_0$ should depend on the sampled $U\sim\mathcal U([0,1])$ and be related to the target distribution with characteristic function $(\phi_\ell)^{\gD_n}$. 
This means we should determine $\widetilde X_0$ implicitly by $F^{(0)}(\widetilde X_0)=U$, where $F^{(0)}$ is a CDF of a distribution similar to the target distribution, but which can be inverted efficiently.
\subsection{Approximation of a GH field}\label{sec:GH_app}
We consider a GH L\'evy field on the (separable) Hilbert space $H=L^2(\mathcal{D})$ with a compact spatial domain $\mathcal{D}\subset\R^s$.  
The operator $Q$ on $H$ is given by a \textit{Mat\'ern covariance operator} with variance $v>0$, correlation length $r>0$ and a positive parameter $\chi>0$ defined by
\begin{equation*}
[Qh](x):=v\int_\mathcal D k_\chi(x,y)h(y)dy,\quad \text{for }h\in H,
\end{equation*}
where $k_\chi$ denotes the \textit{Mat\'ern kernel}.
For $\chi=\frac{1}{2}$, we obtain the exponential covariance function and for $\chi\to\infty$ the squared exponential covariance function. 
For general $\chi>0$, the Mat\'ern kernel 
\begin{equation*}
 k_\chi(x,y):=\frac{2^{1-\chi}}{\gG_G(\chi)}\Big(\frac{\sqrt{2\chi}|x-y|}{r}\Big)^\chi K_\chi\Big(\frac{\sqrt{2\chi}|x-y|}{r}\Big)
\end{equation*}
fulfills the limit identity $k_\chi(x,x)=\lim_{y\to x}k_\chi(x,y)=1$, which can be easily seen by \cite[Eq. (10.30.2)]{NIST10}. 
Here $\gG_G(\cdot)$ is the Gamma function. As shown in~\cite{E09}, this implies 
\begin{equation}\label{eq:trace}
tr(Q)=\sum_{i=1}^\infty\rho_i=v\int_{\mathcal D}dx,
\end{equation}
where $(\rho_i,i\in\N)$ are the eigenvalues of the Mat\'ern covariance operator $Q$.
In general, no analytical expressions for the eigenpairs $(\rho_i,e_i)$ of $Q$ will be available, but the spectral basis may be approximated by numerically solving a discrete eigenvalue problem and then interpolating by \textit{Nystr\"om's method}. 
For a general overview of common covariance functions and the approximation of their eigenbasis we refer to~\cite{RW06} and the references therein.

Now let $L_N^{GH}$ be an approximation of a GH field by a $N$-dimensional GH process $(\ell^{GH_N}(t),t\in\T)$ with fixed parameters $\gl,\ga,\gd\in\R$, $\gb,\mu\in\R^N$ and $\gG\in\R^{N\times N}$. 
The parameters are chosen in such a way that the multi-dimensional GH process has uncorrelated marginal processes, hence the generated KL expansions 
\begin{equation*}
L_N^{GH}(x)(t)=\sum_{i=1}^N\varphi_i(x)\ell^{GH}_i(t)
\end{equation*}
are again one-dimensional GH processes for any spectral basis $(\varphi_i,i\in\N)$ and fixed $x\in\mathcal D$. 
This in turn means, that we may draw samples of $\ell^{GH_N}$ by simulating a GIG process $\ell^{GIG}$ with parameters $a=\gd, b=(\ga^2-\gb'\gG\gb)^{1/2}$ and $p=\gl$ using Fourier inversion and then subordinating a $N$-dimensional Brownian motion (see Lemma~\ref{lem:sub}).
The characteristic function of a GIG L\'evy process $(\ell^{GIG}(t),t\in\T)$ with (fixed) parameters $a,b>0$ and $p\in\R$ is given by
\begin{equation} \label{cf_gig}
\phi_{GIG}(u;a,b,p):=\E[\exp(iu\ell^{GIG}(1))]=(1-2iub^{-2})^{-p/2}\frac{K_p(ab\sqrt{1-2iub^{-2}})}{K_p(ab)}.
\end{equation}
The GIG distribution corresponding to $(\phi_{GIG})^{\gD_n}$ with $\gD_n=1$ is continuous with finite variance (see~\cite{S03}), which implies that these properties hold for all distributions with characteristic function $(\phi_{GIG})^{\gD_n}$, for any $\gD_n>0$. 
The constants as in Assumption~\ref{ass:1} are derived in the following.
For $k\in\N$, the $k$-th moment of the GIG distribution is given as
\begin{equation*}
0<\E\big((\ell^{GIG}(1))^k\big)=\big(\frac{a}{b}\big)^{k}\frac{K_{p+k}(ab)}{K_p(ab)}<\infty.
\end{equation*}
For any $\eta=2k$ we are, therefore, able to calculate the bounding constant $R$ via
\begin{equation*}
R=(-1)^{k}\frac{d^{2k}}{du^{2k}}((\phi_{GIG}(u;a,b,p))^{\gD_n})\big|_{u=0},
\end{equation*}
because the derivatives of $\phi_{GIG}$ evaluated at $u=0$ are 
\begin{equation*}
(\phi_{GIG}(0;a,b,p))^{(k)}=i^{-k}\E\big((\ell^{GIG}(1))^k\big)=i^{-k}\big(\frac{a}{b}\big)^{k}\frac{K_{p+k}(ab)}{K_p(ab)}.
\end{equation*}
The calculation of the $\eta$-th derivative can be implemented easily by using a version of Fa\`{a} di Bruno's formula containing the Bell polynomials, for details we refer to~\cite{J02}.
The bounding constants $\theta$ and $B$ may be determined numerically as described Section~\ref{sec:imp} (e.g. by using the routine \texttt{fminsearch} in MATLAB).
The derivation of the bounds implies that we can ensure $L^p$ convergence of the approximated GIG process in the sense of Theorem~\ref{thm:L^p} for any $p\ge1$, because it is possible to define $\eta$ as any even integer and then obtain $R$ by differentiation.
We observe that the target distribution with characteristic function $(\phi_{GIG}(u;a,b,p))^{\gD_n}$ and $\gD_n>0$ is not necessarily GIG, except for the Inverse Gaussian (IG) case where 
$p=-1/2$ and $(\phi_{IG}(u;a,b))^{\gD_n}=\phi_{IG}(u;\gD_na,b)$(see Remark~\ref{rem:NIG}). 
This special feature of the IG distribution is exploited to determine the initial values $\widetilde X_0$ in the Newton iteration by moment matching: 
Consider an $IG(a_0,b_0)$ distribution with mean $a_0/b_0$ and variance $a_0/b_0^3$, where the parameters $a_0,b_0>0$ are ``matched'' to the target distribution's mean and variance via
\begin{align*}
 \frac{a_0}{b_0}&=i\frac{d}{du}((\phi_{GIG}(u;a,b,p))^{\gD_n})\big|_{u=0},\\
 \frac{a_0}{b_0^3}&=(-1)\frac{d^2}{du^2}((\phi_{GIG}(u;a,b,p))^{\gD_n})\big|_{u=0}-\Big(i\frac{d}{du}((\phi_{GIG}(u;a,b,p))^{\gD_n})\big|_{u=0}\Big)^2.
\end{align*}
If $F_{\gD_n}^{IG}$ denotes the CDF of this $IG(a_0,b_0)$ distribution, the initial value of the globalized Newton method is given implicitly by $F_{\gD_n}^{IG}(\widetilde X_0)=U$. 
The inversion of $F_{\gD_n}^{IG}$ may be executed numerically by many software packages like MATLAB.\\
With our approach, this results in the approximation of a GIG process $\widetilde \ell^{GIG}$ at discrete times $t_j\in\Xi_n$. 
The $N$-dimensional GH process $\ell^{GH_N}$ may then be approximated at $t_j$ for $j=0,\dots,n$ by the process $\widetilde \ell^{GH_N}$ with $\widetilde \ell^{GH_N}(t_0)=0$ and the increments
\begin{equation*}
\widetilde \ell^{GH_N}(t_j)-\widetilde \ell^{GH_N}(t_{j-1})=\mu \gD_n+\gG\gb(\widetilde \ell^{GIG}(t_j)-\widetilde \ell^{GIG}(t_{j-1})) + \sqrt{(\widetilde \ell^{GIG}(t_j)-\widetilde \ell^{GIG}(t_{j-1}))\gG}w^N_j(1),
\end{equation*}
for $j=1,\dots,n$, where the $w_j^N(1)$ are i.i.d. $\cN_N(0,\1_{N\times N})$-distributed random vectors.
To obtain the process $\widetilde\ell^{GH_N}$ at arbitrary times $t\in\T$, we interpolate the samples $(\widetilde\ell^{GH_N}(t_j),j=0\ldots,n)$ piecewise constant as in Algorithm~\ref{algo:approx2}.
With this, we are able to generate an approximation of  $L_N^{GH}$ at any point $(x,t)\in\mathcal D\times\T$ by
\begin{equation*}
\widetilde L_N^{GH}(x)(t):=\sum_{i=1}^N\varphi_i(x) \widetilde\ell^{GH}_i(t).
\end{equation*}
The knowledge of $tr(Q)$ enables us to determine the truncation index $N$ and the constant $C_\ell$ as in Remark~\ref{rem:trunc}:
For $N\in\N$ let $(\widetilde\ell^{GH}_i,i=1,\ldots,N)$ be the approximations of the processes $(\ell^{GH}_i,i=1,\ldots,N)$, where the random vector $(\ell^{GH}_1(1),\dots,\ell_N^{GH}(1))$ is multivariate GH-distributed by assumption.
Hence, for every $N\in\N$, we obtain the parameters $a(N), b(N), \gl(N)$ of a corresponding GIG subordinator $\ell^{GIG,N}$, which is approximated through a piecewise constant process $\widetilde \ell^{GIG,N}$ as above. 
With Eq.~\eqref{L^p error} we calculate the error
\begin{align}\label{eq:E_GIG}
E_{GIG,N}^p:=\sup_{t\in\T}\E(|\ell^{GIG,N}(t)-\widetilde \ell^{GIG,N}(t)|^p).
\end{align}
for $p\in\{1,2\}$.
If $\gb\in\R^N$ and $\gG\in\R^{N\times N}$ denote the GH parameters corresponding to $(\ell^{GH}_1(1),\dots,\ell_N^{GH}(1))$, the $L^2(\gO;\R)$ approximation error of each process $\ell^{GH}_i$ is given by 
\begin{equation*}
\widetilde C_{\ell,i}:=\sup_{t\in\T}\frac{\E(|\ell^{GH}_i(t)-\widetilde\ell^{GH}_i(t)|^2)}{\gD_n}= \frac{E_{GIG,N}^2(\gG\gb)_i^2+E_{GIG,N}^1\sqrt{\gG_{[i]}\gG_{[i]}'}}{\gD_n},
\end{equation*}
where $\gG_{[i]}$ indicates the $i-$th row of $\gG$.
Starting with $N=1$, we compute the first $N$ eigenvalues and the difference
\begin{equation*}
 T\Big(tr(Q)-\sum_{i=1}^N\rho_i\Big)-\max_{i=1,\dots,N}\widetilde C_{\ell,i}\gD_n \sum_{i=1}^N\rho_i
\end{equation*}
and increase $N$ by one in every step until this expression is close to zero.
If a suitable $N$ is found, we define $C_\ell:=\max_{i=1,\dots,N}\widetilde C_{\ell,i}$ and thus have equilibrated truncation and approximation errors by ensuring Eq.~\eqref{eq:trunc}.
For simplicity, we have implicitly assumed here that the processes $\ell^{GH}_i$ were normalized in the sense that $\text{Var}(\ell^{GH}_i(t))=t$.
This is due to the fact that $\rho_i\ell_i$ (here with $\ell_i=\ell^{GH}_i$) in Theorem~\ref{thm:H_error} represents the scalar product $(L(t),e_i)_H$ with variance $\rho_it$. 
In case we have unnormalized processes, one can simply divide $\ell^{GH}_i$ by its standard deviation (see Formula~\eqref{eq:VarZ}) and adjust the constants $\widetilde C_{\ell,i}$ and $C_\ell$ accordingly.
\subsection{Numerical results}
As a test for our algorithm, we generate GH fields on the time interval $\T=[0,1]$ with step size $\gD_n=2^{-6}$, on the spatial domain $\mathcal{D}=[0,1]$. 
For practical aspects, one is usually interested in the $L^1$-error $\E(|\ell(t)-\widetilde\ell^{(n)}(t)|)$ and the $L^2$-error $(\E(|\ell(t)-\widetilde\ell^{(n)}(t)|^2))^{1/2}$.
Upper bounds for both expressions depend on $\eta$ and $D$ and are given by Ineq.~\eqref{L^p error}. 
To obtain reasonable errors, we refer to the discussion on the choice of $D$ in Remark~\ref{rem:t-conv} and set $D=\gD_n^{1/(1-\eta)}$. 
This ensures that the $L^1$-error is of order $\mathcal O (\gD_n)$ and is a good trade-off between simulation time and the size of the $L^2$-error for most values of $\eta$ in the GIG example below.
Choosing for example $D=\gD_n^{2/(2-\eta)}$ would reduce the $L^2$-error to order $\mathcal O (\gD_n)$, but does not have a significant effect on the $L^1$-error and results in a higher computational time.
For the Mat\'ern covariance operator $Q$ we use variance $v=1$, correlation length $r=0.1$ and $\chi\in\{\frac{1}{2},\frac{3}{2}\}$,
where a higher value of $\chi$ increases the regularity of the field along the $x$-direction.
For the fixed GH parameters we choose $\ga=5$, $\gb=\mu=0_N$, $\gd=4$ and $\gG=\1_N$, the shape parameter $\gl$ will vary throughout our simulation and admits the values $\gl\in\{-\frac{1}{2},1\}$, which results in NIG resp. hyperbolic GH fields.
This parameter setting ensures that the multi-dimensional GH distribution has uncorrelated marginals, hence the truncated KL expansion $L_N^{GH}$ of $L^{GH}$ is itself an infinite dimensional GH L\'evy process.
Further, for every $N\in\N$, the constant $\widetilde C_{\ell,i}$ from Section~\ref{sec:GH_app} is independent of $i=1,\dots,N$, thus the truncation index $N$ can easily be determined to balance out the Fourier inversion and truncation error for each combination of $\gl$ and $\chi$. 
To examine the impact of $\eta$ on the efficiency of the simulation, we set $\eta\in\{4,6,8,10\}$ and the constant $R$ as suggested in Section~\ref{sec:GH_app} for each $\eta$.
For fixed $\eta$ and $R$, we choose $\theta\in\{1,1.5\dots,99.5,100\}$ and calculate for each $\theta$ the constant $B_\theta$ as in Section~\ref{sec:imp}. 
This results in up to $199$ different values for the number of summations $M_\theta$, which all guarantee the desired accuracy $\eps$, meaning we can choose the smallest $M_\gt$ for our simulation.
The optimal value $\gt_{opt}$ which leads to the smallest $M_\gt$ depends highly on the GH parameters and may vary significantly with $\eta$. 
	For $\gl=1$, we found that $\gt_{opt}$ ranges from $34$ to $68.5$, varying with each choice of $\eta\in\{4,6,8,10\}$. 
	In contrast, in the second example with $\gl=-1/2$, $\gt_{opt}=11$ independent of $\eta$.
We generate 1.000 approximations $\widetilde L_N^{GH}$ for several combinations of $\gl$, $\chi$ and $\eta$, which allows us to check if the generated samples actually follow the desired target distributions. 
To this end, we conduct Kolomogorov--Smirnov tests for the subordinating GIG process as well as for the distribution of the GH field at a fixed point in time and space and report on the corresponding p-values.
\begin{figure}
\centering
    \subfigure[Sample of a GH field]{\includegraphics[scale=0.45]{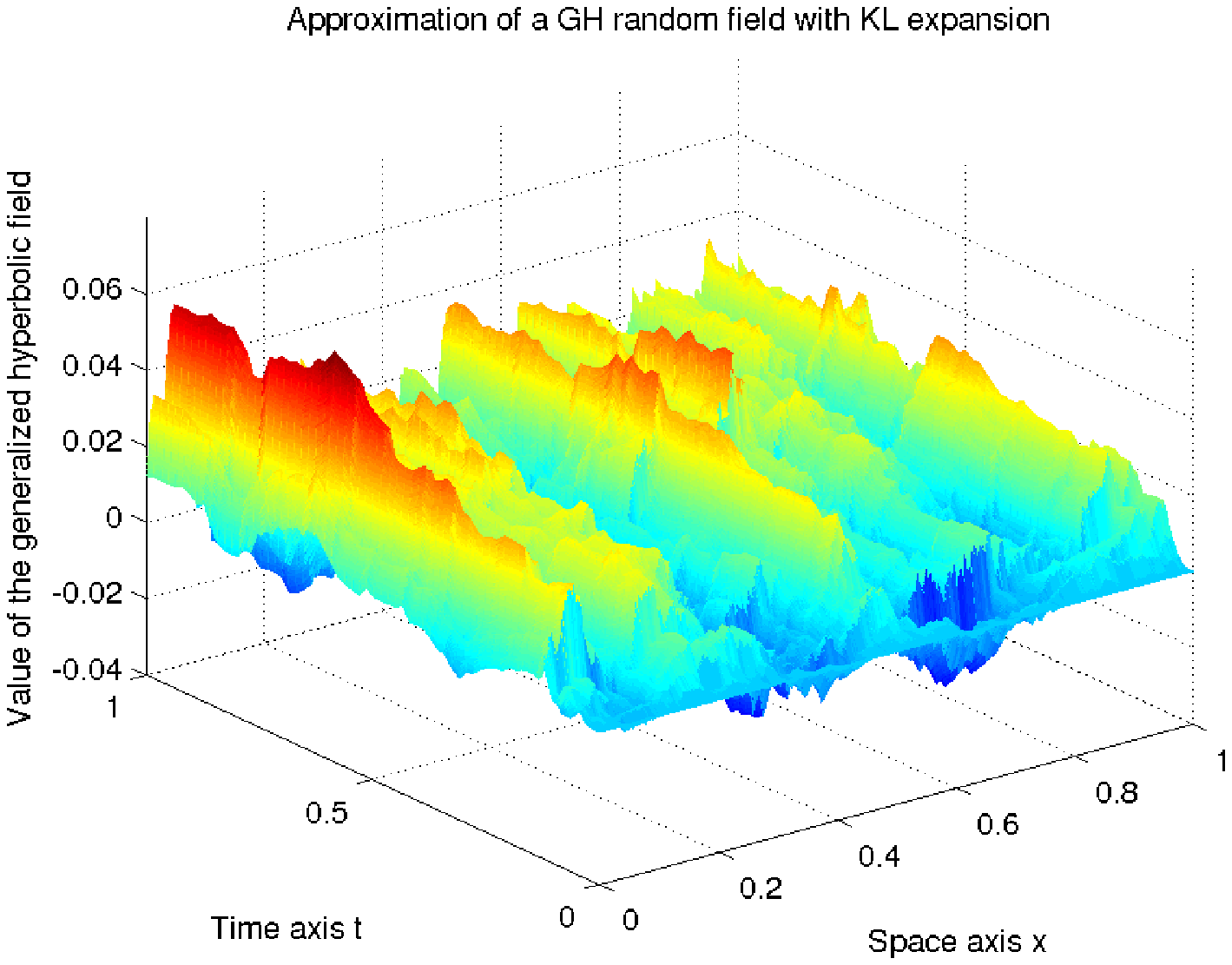}}
    \subfigure[Empirical dist. of 1.000 samples at $t=x=1$]{\includegraphics[scale=0.48]{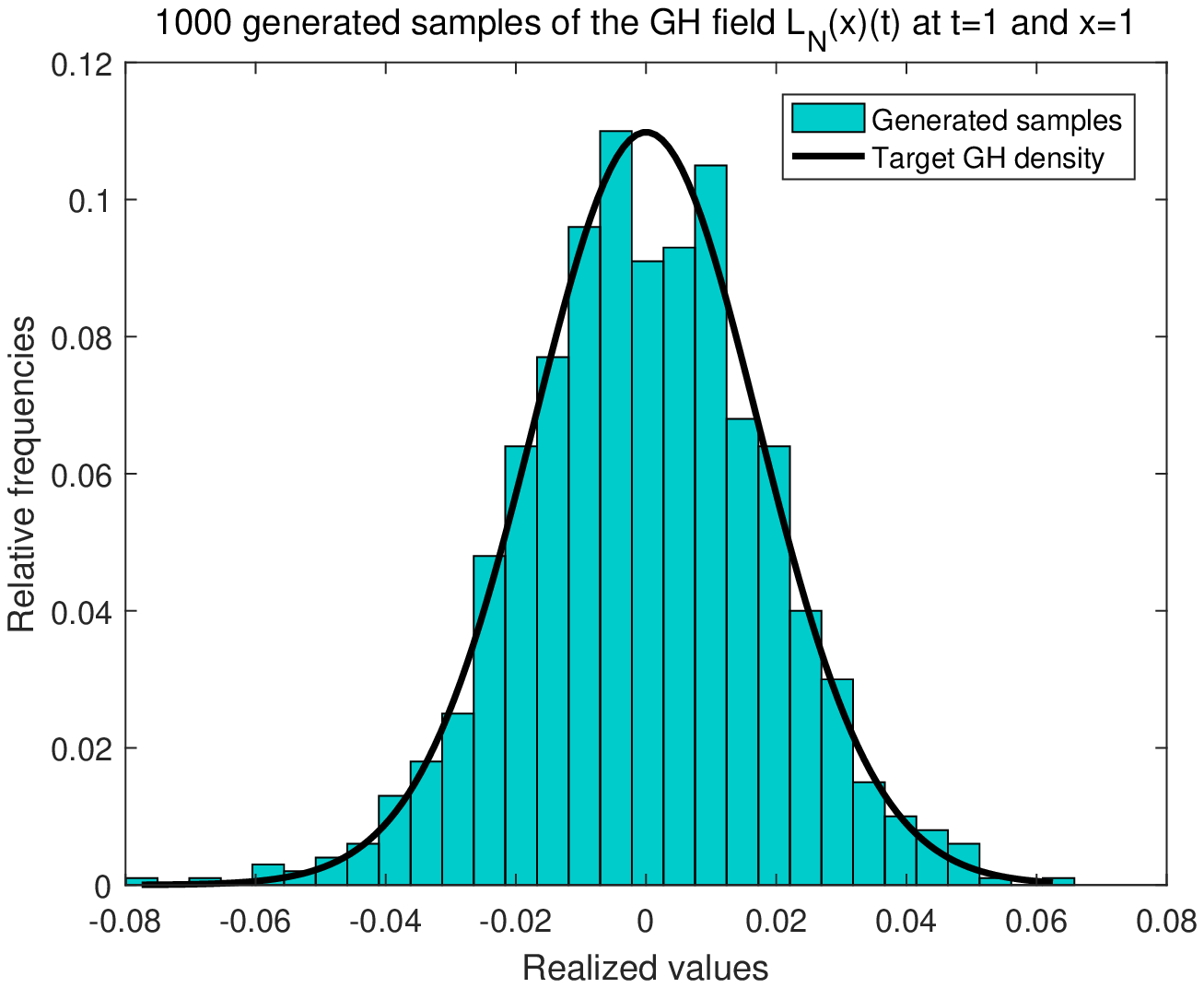}}
\caption{Sample and empirical distribution of an hyperbolic field with parameters  $\gl=1$, $\chi=1/2$, $\eta=10$ and truncation after $N=132$ terms.}
\label{fig:hyp}
\centering
    \subfigure[Sample of a GH field]{\includegraphics[scale=0.45]{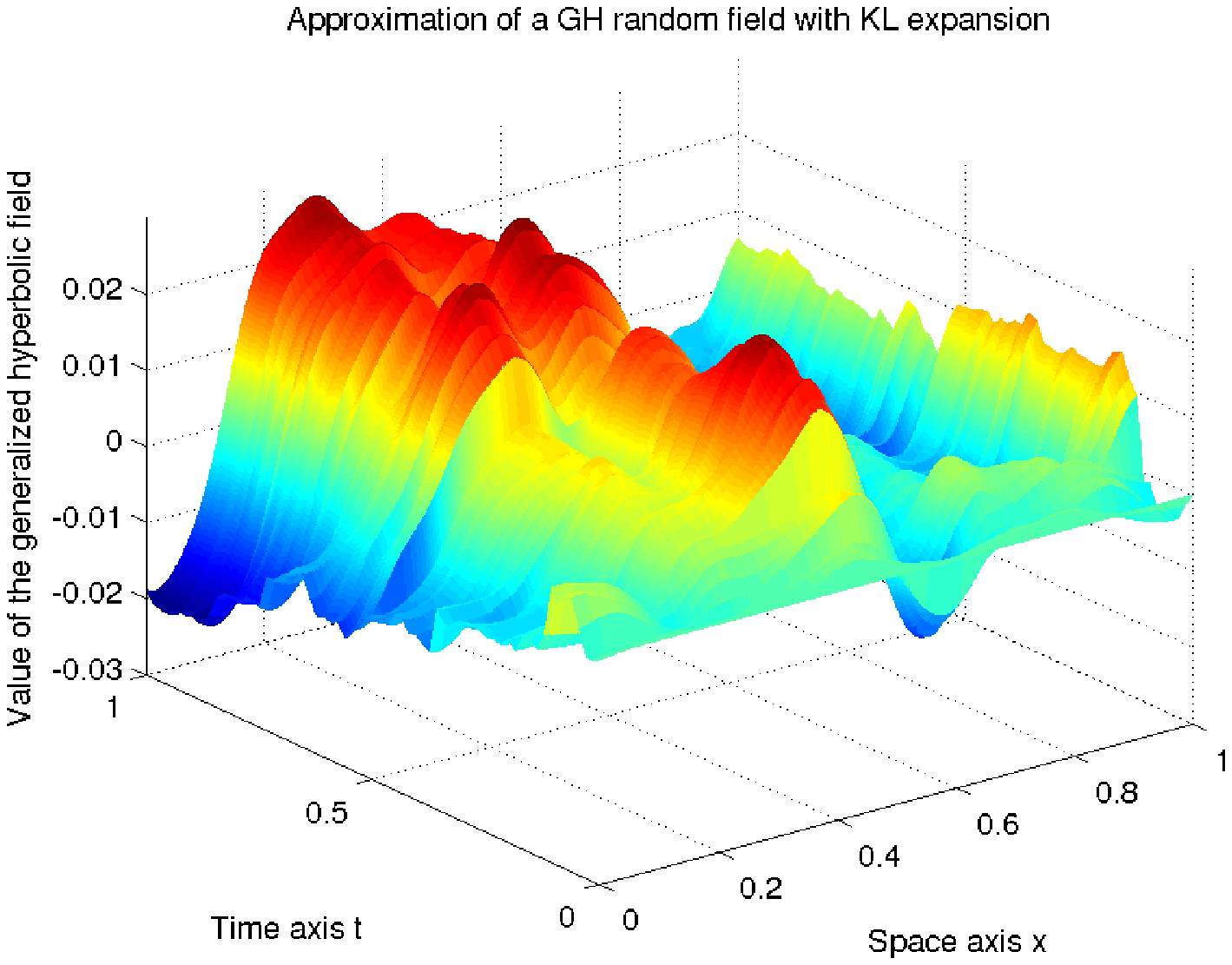}}
    \subfigure[Empirical dist. of 1.000 samples at $t=x=1$]{\includegraphics[scale=0.48]{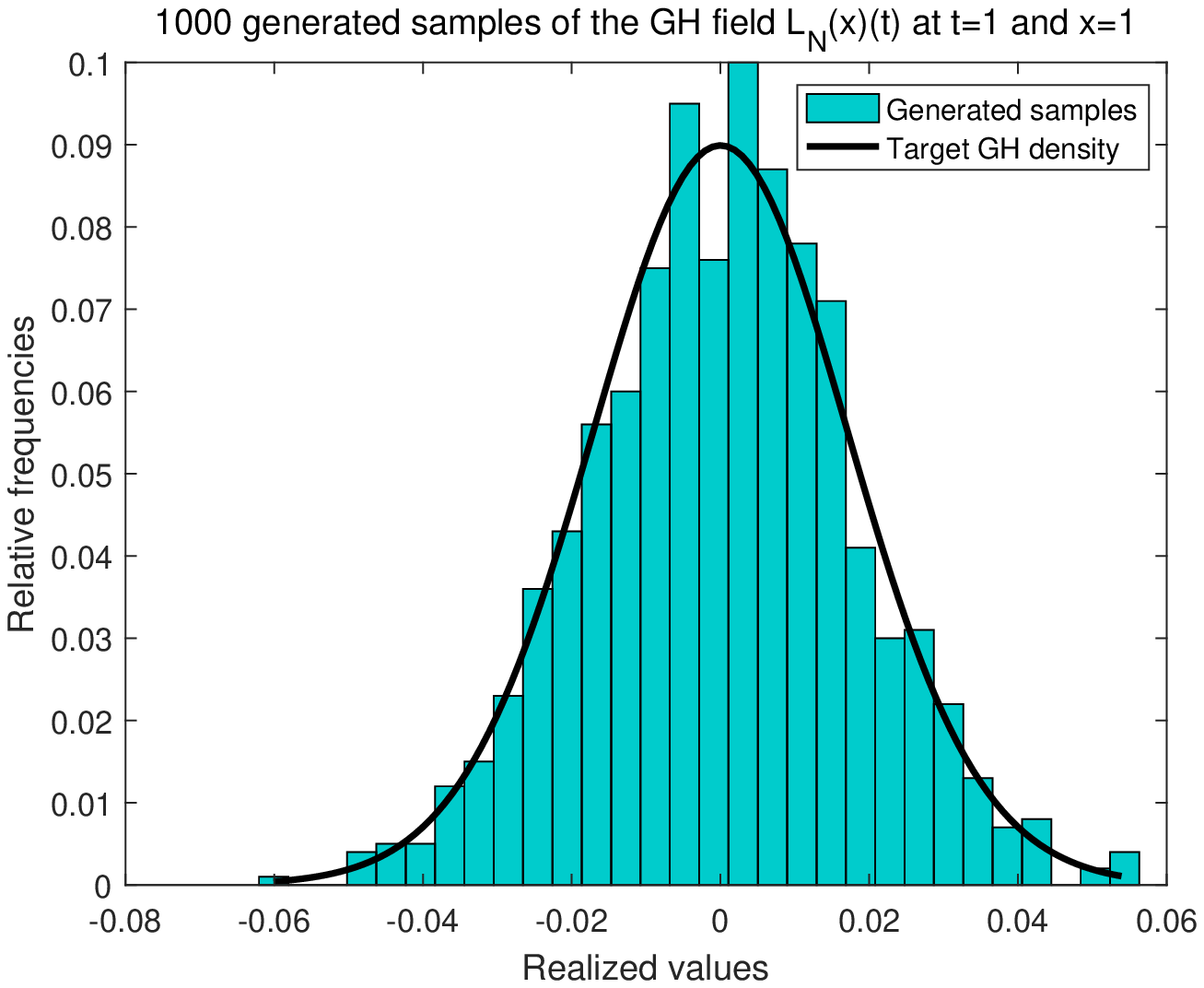}}
\caption{Sample and empirical distribution of a NIG field with parameters  $\gl=-1/2$, $\chi=3/2$, $\eta=10$ and truncation after $N=18$ terms.}
\label{fig:nig}
\end{figure}
Figures~\ref{fig:hyp} and ~\ref{fig:nig} show samples of approximated GH random fields:
Along the time axis we see the characteristic behavior of the (pure jump) GH processes for every point $x\in\mathcal D$. 
For a fixed point in time $t$, the paths along the $x$-axis vary according to their correlation, depending on the covariance parameter $\chi$.
As reported in \cite{RW06}, the eigenvalues of $Q$ decay slower if $\chi$ becomes smaller, meaning we need a higher number of summations $N$ in the KL expansion so that the error contributions are equilibrated. 
This effect can be seen in Tables~\ref{tab:hyp} and~\ref{tab:nig}, where the truncation index $N$ changes significantly with $\chi$.
If the KL expansion, however, can be sampled by a $N$-dimensional GH process as suggested in Theorem~\ref{thm:Z_U},
the number of summations $N$ has only a minor impact on the computational costs of the KL expansion.  
This is due to the fact that in this case the time consuming part, namely simulating the subordinator, has to be done only once, regardless of $N$. 
Compared to these costs, the costs of subordinating a Brownian motion of any finite dimension are negligible.
The histograms in Figures~\ref{fig:hyp} and ~\ref{fig:nig} show the empirical distribution of the approximation $\widetilde L_N^{GH}(x)(t)$ at time $t=1$ and $x=1$. 
The theoretical distribution at time 1 and an arbitrary point $x\in\mathcal D$ is again GH, where the parameters are given in Lemma~\ref{lem:KL_GH}.
Obviously, the empirical distributions fit the target GH distributions from Lemma~\ref{lem:KL_GH}. 
To be more precise, we have conducted a Kolmogorov-Smirnov test for both, the subordinating GIG process and the GH field at time $t=1$ and for the latter at $x=1$. 
We know the law of both processes at $x\in\cD$ and are able to obtain their CDFs sufficiently precise for the tests by numerical integration. 
The test results for 1.000 samples of the hyperbolic resp. the NIG field with covariance parameters $\chi=\frac{1}{2}$ resp. $\chi=\frac{3}{2}$ are given in Tables~\ref{tab:hyp} and~\ref{tab:nig} above and do not suggest that the generated samples follow another distribution than the expected one.

\begin{table}

\begin{center}
\small
\begin{tabular}{l|*{3}{c|}c} 
$\eta$  & $E_{GIG,N}^1$ & $E_{GIG,N}^1/\gD_n$ &$E_{GIG,N}^2$&$\E[||L^{GH}(1)-\widetilde L^{GH}_N(1)||_H^2]$ \\ \hline
4 & 0.0143 &0.9166& 0.2584 &0.0646\\
6 & 0.0138 &0.8835& 0.0749 & 0.0635\\
8 & 0.0138 &0.8824& 0.0601 &0.0634\\
10& 0.0140 &0.8975& 0.0806 &0.0636\\ \hline \hline
$\eta$  &$N$ & p-value GH&abs. time&rel. time\\ \hline
4 &130& 0.8246 & 0.1945 sec.&100.00\% \\
6 &133& 0.3077 & 0.1093 sec.&56.19\% \\
8 &133& 0.3077 & 0.0851 sec. &43.78\%\\
10&132& 0.2873 & 0.0759 sec. &39.04\%\\ \hline
\end{tabular}
\end{center}
\caption{Errors, p-values and average simulation times per field based on 1.000 simulations.
Stepsize $\Delta t =2^{-6}$ and  $D=\Delta t^{1/(1-\eta)}$. 
GH process: $\gl=1$, $\alpha = 5, \beta = 0_N$, $\delta = 4, \mu = 0_N, \Gamma = 1_{N\times N}$.
Covariance parameters: $\chi=1/2$, $r=0.1$ and $v=1$.
The KS test for the GIG subordinator returns a p-value of 0.5498 for each $\eta\in\{4,6,8,10\}$.}
\label{tab:hyp}
\begin{center}
\small
\begin{tabular}{l|*{3}{c|}c} 
$\eta$  & $E_{GIG,N}^1$ & $E_{GIG,N}^1/\gD_n$ &$E_{GIG,N}^2$&$\E[||L^{GH}(1)-\widetilde L^{GH}_N(1)||_H^2]$ \\ \hline
4 & 0.0132 & 0.8443 & 0.2079 & 0.0619   \\
6 & 0.0128 & 0.8170 & 0.0584 & 0.0608 \\
8 & 0.0127 & 0.8155 & 0.0456 & 0.0608\\
10& 0.0129 & 0.8252 & 0.0589 & 0.0611  \\ \hline \hline
$\eta$  &$N$ & p-value GH&abs. time&rel. time\\ \hline
4 & 18 & 0.9223 & 0.1039 sec. & 100.00\% \\
6 & 18 & 0.9223 & 0.0628 sec. & 60.43\% \\
8 & 18 & 0.9223 & 0.0460 sec. & 44.29\% \\
10& 18 & 0.9223 & 0.0380 sec. & 38.59\% \\ \hline 
\end{tabular}
\end{center}
\caption{Errors, p-values and average simulation times per field based on 1.000 simulations.
Stepsize $\Delta t =2^{-6}$ and  $D=\Delta t^{1/(1-\eta)}$. 
GH process: $\gl=-1/2$, $\alpha = 5, \beta = 0_N$, $\delta = 4, \mu = 0_N, \Gamma = 1_{N\times N}$.
Covariance parameters: $\chi=3/2$, $r=0.1$ and $v=1$.
The KS test for the GIG subordinator returns a p-value of 0.6145 for each $\eta\in\{4,6,8,10\}$.}
\label{tab:nig}

\end{table}
We denote by $E_{GIG,N}^1$ and $E_{GIG,N}^2$ the approximation error of the subordinator as in Eq.~\eqref{eq:E_GIG}, which we have listed in absolute terms in Tables~\ref{tab:hyp}~and~\ref{tab:nig}. The first error bound is also given relative to $\gD_n$ to show that it is in fact of magnitude $\mathcal O(\gD_n)$.
While the $L^1(\gO;\R)$-error $E_{GIG,N}^1$ is relatively constant for each $\eta$, the $L^2(\gO;\R)$-error $E_{GIG,N}^2$ is rather high for $\eta=4$, but has an acceptable upper bound for $\eta\ge6$. 
This is not surprising, since $D=\gD_n^{1/(1-\eta)}$ only guarantees that $\E(|\ell^{GIG}(t)-\widetilde \ell^{GIG}(t)|)=\mathcal O(\gD_n)$.
We emphasize that the (theoretic) error bounds in Tables~\ref{tab:hyp}~and~\ref{tab:nig} are very conservative as the triangle inequality and similar "coarse" estimates were used repeatedly in their estimation in Theorem~\ref{thm:L^p 1} and~\ref{thm:L^p}.
The truncation index $N$ is highly sensitive to $\chi$, but has small or no variations for fixed $\chi$ and varying $\eta$. 
Since we choose $t\in[0,1]$, the expression $\E(||L^{GH}(1)-\widetilde L^{GH}_N(1)||_H^2)$ in Tables~\ref{tab:hyp}~and~\ref{tab:nig} is an upper bound for the $L^2(\gO;H)$-error $\sup_{t\in[0,1]}\E(||L^{GH}(t)-\widetilde L^{GH}_N(t)||_H^2)$. 
Note that this error is small in relative terms, since by our choice of $Q$ and Eq.~\ref{eq:trace} we have 
$\E(||L^{GH}(1)||^2_H)=tr(Q)=1$. 

The p-value of the GH distribution varies if different $N$ are chosen for the KL expansion, which is natural due to statistical fluctuations. 
More importantly, the null hypothesis, namely that the samples follow a GH distribution with the expected parameters, is never rejected at a $5\%$-level.
As expected, the speed of the simulation heavily depends on $\eta$.

Looking at the results for $\eta=4$, one might argue that the Fourier inversion method is only suitable for processes where this parameter can be chosen rather high, i.e. for distributions which admit a large number of finite moments.
To qualify this objection, we consider once more the t-distribution with three degrees of freedom and the corresponding L\'evy process $\ell^{t3}$ from Example~\ref{ex:eta_cond}. 
Since $\E(\ell^{t3}(\gD_n))=0$ and $\text{Var}(\ell^{t3}(\gD_n))=\sqrt3\gD_n$, we can choose $\eta=2$ and hence $R=\sqrt3\gD_n$. 
The characteristic function of $\ell^{t3}(\gD_n)$ is given by
\begin{equation*}
(\phi_{t3}(u))^{\gD_n}=\exp(-\sqrt3\gD_n|u|)(\sqrt3|u|+1)^{\gD_n}
\end{equation*}
and $B$ and $\theta$ are estimated in the same way as for the GIG process. 
Using again $\gD_n=2^{-6}$ and $D=\gD_n^{1/(1-\eta)}$, we obtain that the number of summations in the approximation is $M=12.924$ for $\theta=\frac{19}{2}$.
The simulation time for one process $\widetilde\ell^{t3}$ with $(\gD_n)^{-1}=2^6$ increments in the interval $[0,1]$ is on average 0.0655~seconds, where the initial values have been approximated by matching the moments of a normal distribution
(the Kolmogorov-Smirnov test for a t-distribution at $t=1$ based on 1.000 samples returns a p-value of $0.5994$).
In the GIG example, we needed $M=79.086$ terms in the summation if $\eta=4$ is chosen and still $M=33.030$ terms for $\eta=10$.
This shows that the Fourier Inversion method is also applicable if $\eta$ can only be chosen relatively low and that the GIG (resp. GH) process is a computationally expensive example of a L\'evy process.

\bibliographystyle{siam}
\bibliography{Levy_Fields}

\end{document}